\documentclass[11pt]{article}

\usepackage{amssymb}
\usepackage{amsthm}
\usepackage{amsmath}
\usepackage{amscd}
\usepackage{graphicx}
\usepackage{hyperref}

\allowdisplaybreaks
    
\newcommand{\R}{\mathbb{R}}

\newcommand{\N}{\mathbb{N}}
\newcommand{\Sph}{\mathbb{S}}

\newcommand{\eps}{\varepsilon} 

\newcommand{\me}{\mathrm{e}}

\newcommand{\dif}{\mathrm{d}}

\newcommand{\vol}{\mathrm{Vol}}

\newcommand{\coan}{C^{0,\alpha}_{\nu - 2}}
\newcommand{\coann}{C^{0,\alpha}_{\nu - 2, \bar \nu}}
\newcommand{\ctan}{C^{2,\alpha}_\nu}
\newcommand{\ctann}{C^{2,\alpha}_{\nu, \bar \nu}}

\newcommand{\riem}{\mathrm{Rm}}
\newcommand{\ric}{\mathrm{Ric}}

\newcommand{\apsol}{\tilde \Sigma_r(\sigma, \delta)}
\newcommand{\apsolf}{\tilde \Sigma_r^{\mathit{F}}(\sigma, \delta)}
\newcommand{\apsoloe}{\tilde \Sigma_r^{\mathit{OE}}(\sigma, \delta)}

\DeclareMathOperator{\arccosh}{arccosh}

\newtheorem{thm}{Theorem}
\newtheorem{lemma}[thm]{Lemma}

\newtheorem{prop}[thm]{Proposition}
\newtheorem*{nonumthm}{Theorem}

\theoremstyle{definition}
\newtheorem{defn}[thm]{Definition}

\newtheoremstyle{rmk}{5pt}{5pt}{}{}{\scshape}{:}{.5em}{}
\theoremstyle{rmk}
\newtheorem*{rmk}{Remark}

\newcommand{\mylabel}
	{\label}
\begin{document}

\title{CMC Hypersurfaces Condensing to \\ Geodesic Segments and Rays \\
in Riemannian Manifolds} 

\author{Adrian Butscher \thanks{butscher@math.stanford.edu, Department of Mathematics, Stanford University, Stanford, CA 94305} \\ Stanford University \and Rafe Mazzeo \thanks{mazzeo@math.stanford.edu,  Same mail address. Supported by the NSF Grant DMS-0805529} \\ Stanford University}

\maketitle

\begin{abstract}
We construct examples of compact and one-ended constant mean curvature surfaces with large mean curvature in Riemannian manifolds with axial symmetry by gluing together small spheres positioned end-to-end along a geodesic. Such surfaces cannot exist in Euclidean space, but we show that the gradient of the ambient scalar curvature acts as a `friction term' which permits the usual analytic gluing construction to be carried out.
\end{abstract}

\renewcommand{\baselinestretch}{1.25}
\normalsize

\section{Introduction}
\paragraph*{Background.}
The study of constant mean curvature (hereafter CMC) surfaces in $\R^3$, or more generally in three-dimensional Riemannian manifolds, is a well established field of Riemannian geometry and the literature concerning the construction and properties of such surfaces is enormous.  One particular method for constructing CMC surfaces is by analytic gluing techniques.  These go back to the work of Kapouleas \cite{kapouleas7},\cite{kapouleas6} and have been further developed by many others, including the first author with Pacard \cite{mepacard1},\cite{mepacard2} and the second author with Pacard and also with Pollack \cite{mazzeopacardends},\cite{mazzeopacardpollack}. (See \cite{mazzeosurvey} and \cite{pacardsurvey} for surveys about the current  state of this approach.)  The general idea here is to take connected sums of simple surfaces (for instance the classical examples of nonminimal CMC surfaces in $\R^3$ --- the sphere, the cylinder, and the one-parameter family of Delaunay surfaces of revolution) to produce more general CMC surfaces of finite topology, both compact and non-compact.  Such constructions work by finding a surface whose mean curvature is nearly CMC embedding of the connected sum and perturbing it to have constant mean curvature.  

There are numerous constraints on the structure of CMC surfaces of finite topology in $\R^3$. For example, Meeks \cite{meeks} proved that any end of a complete Alexandrov-embedded CMC surface in $\R^3$ is cylindrically bounded; while Korevaar, Kusner and Solomon \cite{kks} improved this by showing that any such end converges exponentially to one end of a Delaunay surface.  Furthermore, the possible directions of the axes of these ends are also subject to limitations, as well as the flexibility to change these directions within  the moduli space of all such surfaces, see \cite{gks},\cite{gks2} and \cite{gbkkrs}.  These limitations are phrased in terms of a certain \emph{flux} that was discovered by Kusner. The flux is a vector associated to any closed loop in a CMC surface; it is constant under deformations of this loop, and in fact only depends on the homology class of this loop in the surface. There is a flux associated to a simple positively oriented loop around each asymptotically Delaunay end, which depends only the direction of the axis and the neck-size of the limiting Delaunay surface. The homological invariance also shows that the sum over all ends of these limiting fluxes must vanish, which is a global balancing condition for the entire CMC surface. One immediate consequence is the non-existence of a complete Alexandrov-embedded CMC surface in $\R^3$ with only one end. The flux also provides useful local information. For instance, it is a crucial ingredient in the gluing constructions mentioned above, since as we explain more carefully below, one must choose the initial approximate CMC configurations so that the fluxes are almost constant across the necks of the connected sums. 

We turn now to a newer theme in this subject, namely the study of sequences of CMC surfaces with mean curvature tending to infinity in an arbitrary $3$-manifold. Examples in $\R^3$ include sequences of dilations of the complete CMC surfaces with $k$ asymptotically Delaunay ends.  Such sequences of surfaces `condense' onto one-dimensional sets, here a union of half-lines meeting  at a point. This seems to be a general phenomenon: Rosenberg \cite{rosenberg} has shown that if $\Sigma$ is a closed CMC surface in an arbitrary (compact) $3$-manifold $M$, with sufficiently large mean curvature $H$, then $M \setminus \Sigma$ has two components, and the inradius at any point in one of these components is bounded above by $C/H$. In other words, $\Sigma$ looks like a tube around some (presumably $1$-dimensional) set $\gamma$. If $\Sigma_j$ is any sequence of CMC surfaces with mean curvature $H_j \to \infty$ condensing to a curve $\gamma$, then a formal calculation (assuming that the supremum of the pointwise norm of the second fundamental form is of the same order as $H_j$) shows that $\gamma$ is a geodesic, or at least a union of geodesic arcs. This leads to the following central question.

\begin{quote}
{\bf Question:} What are the possible condensation sets $\gamma$ in a $3$-manifold $M$ for sequences of
CMC surfaces $\Sigma_j$ with mean curvatures $H_j \nearrow \infty$.
\end{quote}

The obvious guess is that a condensation set is some sort of network of geodesics. Based on the examples of dilated CMC surfaces in $\R^3$, one expects each edge of this geodesic network to have a `weighting' which carries information about the Delaunay parameters of the CMC tubular piece which converges to that edge. This is far from being proved, but there are some very partial results.  Motivating this conjecture and as a first step on it, the second author and Pacard \cite{mazzeopacardtubes} proved that if $\gamma$ is any closed geodesic in $M$ which is non-degenerate (in the sense that its Jacobi operator is invertible), then most geodesic tubes of sufficiently small radius about $\gamma$ can be perturbed to CMC surfaces with large $H$.  The present paper undertakes a next step toward this conjecture. We prove that for very special (non-compact) ambient $3$-manifolds, there do exist sequences of one-ended CMC surfaces condensing to geodesic rays and sequences of compact CMC surfaces condensing to geodesic intervals. However, for reasons we explain below, it is not clear whether the fact that these limiting curves are geodesics is their most relevant feature.  Our construction works in higher dimensions too, i.e.\ we construct sequences of CMC hypersurfaces condensing to a ray or interval. 

There is an important feature of this condensation question which has not been mentioned yet. The first result about CMC surfaces of high mean curvature was due to Ye \cite{ye} in the early 1990's.  He proved that if $p$ is a non-degenerate critical point of the scalar curvature $S$ of $M^{n+1}$, then geodesic spheres around $p$ with small radius may be perturbed to CMC surfaces with large $H$. He also showed the converse: assuming bounded eccentricity, sequences of CMC spheres with $H \to \infty$ converge to a point $p$ where $\nabla S(p) = 0$.  A more recent paper by Pacard and Xu \cite{pacardxu} considers the same problem in manifolds with constant scalar curvature and proves that there is a secondary curvature function whose critical points regulate the location of these small CMC spheres.  Therefore the role of the scalar curvature of the ambient manifold in the question of CMC surfaces condensing to one-dimensional sets must be addressed.

The CMC surfaces constructed in this paper are perturbations of collections of small spheres joined together by even smaller catenoidal necks, all arranged along a curve $\gamma$.  The sizes of the catenoidal necks are quite small compared to the radii of the spheres, but quite strikingly, these neck-sizes must vary along this chain of spheres.  Indeed, it is precisely because these neck-sizes decrease that the surface eventually `caps off' to an end rather than continuing.  This phenomenon is caused by a flux formula that involves the gradient of the scalar curvature of the ambient manifold.  Indeed, unlike in Euclidean space, the difference of the fluxes computed on two loops which are close to one another need not vanish, but may be computed in terms of a surface integral involving $\nabla S$. (This is completely analogous to the generalized Pohozaev identity discovered by Schoen which arises in his construction of metrics of constant positive scalar curvature \cite{schoen}.) In our setting, this shows that the difference between successive neck-sizes can be expressed in terms of the gradient of the scalar curvature along the axis connecting these two necks.  Thus, in some sense, $\nabla S$ acts as a friction term.  By contrast, the scalar curvature of $M$ does not play a role in the location of the $1$-dimensional condensation set in \cite{mazzeopacardtubes}; the only important feature there is that the condensation set is a closed non-degenerate geodesic. This is almost surely because the CMC surfaces constructed there are nearly cylindrical. 

We now describe this more carefully. Let $\Sigma$ be a hypersurface with constant mean curvature $H$ in $(M^{n+1}, g)$. Suppose that $\mathcal U$ and $\mathcal W$ are open sets in $\Sigma$ and $M$, respectively, such that $\partial \bar{\mathcal W} = \bar{\mathcal U} \cup Q$ for some hypersurface-with-boundary $Q$. If there happens to exist a Killing field $V$ on $M$, then the first variation formula for the area of $\mathcal U$ with the volume of $\mathcal W$ fixed relative to the one-parameter family of diffeomorphisms for $V$ gives that 
\begin{equation}
	\mylabel{eqn:balancing}
	\int_{\partial \mathcal U} g(\nu, V) - H \int_Q g(N, V) = 0;
\end{equation}
here $\nu$ is the unit normal vector field of $\mathcal V$ in $\Sigma$ and $N$ is the unit normal vector field of $Q$ in $M$. The flux itself is defined as 
\begin{equation}
	\mylabel{eqn:flux}
	\int_{\gamma} g(\nu, V) - H \int_Q g(N, V)
\end{equation}
where $\gamma$ is a curve in $\Sigma$ and $Q$ is any surface in $M$ with $\partial Q = \gamma$; this integral is independent of the choice of $Q$.  Thus \eqref{eqn:balancing} is the statement that the flux depends only on the homology class of $\gamma$ in $\Sigma$.  

As already noted, these flux integrals determine when an approximately CMC surface can be perturbed to be exactly CMC. This can be explained more concretely as follows.  Suppose first that $M = \R^{n+1}$ and let $\Sigma$ consist of a collection of spheres of radius $r$ (hence mean curvature $n/r$) connected to each other by small catenoidal necks. Let $\mathcal U$ be one of these spheres with two small spherical caps removed where the necks are attached, $Q$ the union of two disks capping these boundaries and $\mathcal W$ the slightly truncated ball enclosed by $\mathcal U \cup Q$.  Then \eqref{eqn:balancing} becomes
\begin{equation}
	\label{eqn:balancingone}
	\int_{\partial \mathcal U} g(\nu, V) - \frac{n}{r} \int_Q g(N, V) = 
\sum_{\mbox{\scriptsize all necks}} r^{n-1} \eps_i V_i + \mathcal O(r^{n-1} \eps^2),
\end{equation}
where $V_i$ is the unit vector pointing from the center of sphere in question to the $i^{\mathrm{th}}$ neck, 
$r \eps_i$ is the width of this neck and $\eps := \max_i \{ \eps_i \}$.  If $\Sigma$ were exactly CMC, the 
left hand side would necessarily vanish.  If $\Sigma$ is not exactly CMC, then it is a fundamental fact that in order to find a nearby CMC surface, it suffices that the leading term on the right hand side of \eqref{eqn:balancingone} must vanish for each spherical region $\mathcal U$ in $\Sigma$.  If this condition is satisfied, the approximate CMC surface $\Sigma$ is called \emph{balanced}.  Note that it is impossible to have a balanced approximately CMC surface where some sphere has only one spherical neighbour. 

When the ambient manifold $(M^{n+1},g)$ is arbitrary, one can form approximate CMC surfaces with large $H$ by attaching together some large collection of geodesic spheres of very small radius $r$. There are (usually) no Killing fields, but we can use the approximate Killing fields corresponding to translations and rotations in Riemann normal coordinates based at the center of any one of these spheres. Formula \eqref{eqn:balancing} now becomes
\begin{equation}
	\label{eqn:balancingtwo}
	\int_{\partial \mathcal U} g(\nu, V) - H \int_Q g(N, V) = 
\sum_{\mbox{\scriptsize all necks}} r^{n-1} \eps_i V_i  - C r^{n+2} \nabla S (p) + \mathcal O(r^{n-1} \eps^2) + \mathcal O(r^{n+4})
\end{equation}
where $\nabla S$ is the gradient of the scalar curvature of $M$ and $C$ is some explicit dimensional constant.  Note that when applied to a sphere with no neighbours, hence with all $\eps_i = 0$ by default, this gives Ye's condition that the right hand side of \eqref{eqn:balancingtwo} must vanish like $r^{n+4}$. 

The main point in this paper is to exploit the contribution of $\nabla S$ in \eqref{eqn:balancingtwo}. Spheres 
of radius $r$ are joined by necks of width $r \eps$, where by this same formula it is natural to assume that $\eps = \mathcal O(r^3)$; these configurations are arranged in such a way that the leading term on the right in \eqref{eqn:balancingtwo} vanishes.  The perturbation argument producing a nearby CMC surface is then not so different than the one in Euclidean space.
  
\paragraph*{Description of the surfaces.}  
Two specific examples of CMC surfaces in $M$ exhibiting markedly different properties from those 
occuring in Euclidean space will be produced in this paper. We shall make extremely strong assumptions
about the geometry of $(M,g)$ in order to simplify the calculations, which even so are still quite lengthy. Thus our result should be regarded as a model for what should happen in more general manifolds, though that would take considerably more work. Our surfaces will consist of geodesic spheres of small radius $r$ arranged along a geodesic segment or ray $\gamma \subset M$ and joined by suitably scaled pieces of catenoids. In the first example, some large number of spheres, on the order of $1/r$, are glued together so that the resulting surfaces are embedded and compact; the second is embedded and complete, and is built from some large number of spheres joined at one end to a half  Delaunay surface of small neck size. In either case, there is a terminal spherical component which has only one spherical neighbour.

We make the following assumptions about $(M,g)$.  First, let $\gamma$ be a geodesic segment or ray in
$M$ and assume that there is a neighbourhood of $\gamma$ in which the metric $g$ is axially symmetric, 
i.e.\ invariant with respect to rotations about the axis $\gamma$. Thus Fermi coordinates around $\gamma$
identify this tubular neighbourhood with $[0,L] \times D$ (where $L = \infty$ is allowed and $D$ is a disk in $\R^{n-1}$), and
\begin{equation}
	\label{eqn:metricform}
	g = dt^2 +  A(t) \delta;
\end{equation}
here $\delta$ is the standard Euclidean metric on $D$ and $A(t)$ is a smooth, strictly positive function of the arclength $t$ along $\gamma$. The scalar curvature of $g$ is $S :=  A^{-2} (- n A \, \dot A  + \frac{2 - (n-1)(n-2)}{4} \dot A^2)$.  Further assumptions on $A$ depend on whether we wish a finite-length or one-ended CMC submanifold.

\begin{enumerate}
	\item In order to construct a finite-length surface, assume that $A$ is an even function of $t$ so that 
the reflection $t \to -t$ induces an isometry of the tubular neighbourhood of $\gamma$. Assume furthermore 
that $t=0$ is a non-degenerate local maximum of $S$.
	\item In order to construct a non-compact one-ended surface, assume that when $t > 0$, the scalar curvature 
is negative and increases monotonely to $0$, and that $|S(t)| \leq C \me^{\alpha t}$ for some $\alpha< 0$. (One function which satisfies this is $A(t) := 1 + \me^{-t}$.)
\end{enumerate}

These assumptions significantly reduce the complexity of the perturbation argument, but allow for the one feature which allows for this new behaviour of CMC surfaces, namely that $\nabla S$ points along $\gamma$. By our assumptions, however, the geodesic $\gamma$ is also an integral curve for $\nabla S$, and it is unclear which of these two features is the crucial one. One basic question we leave open is the geometric characterization of these condensation curves in more general ambient geometries. We expect new and interesting behaviour to occur when $\nabla S$ no longer is required to point along geodesics.

Our main result can be expressed as follows.
\begin{nonumthm}
	Let $M$ be a Riemannian manifold with the special features described above.
	\begin{itemize}

	\item Let $I$ be a finite segment of the geodesic $\gamma$, where the parametrization of $\gamma$ is such that the reflection $t \mapsto -t$ is an isometry in some neighbourhood of $I$. Then there exists an $r_0 > 0$ so that for every $0 < r < r_0$, there is a CMC surface $\Sigma^F_r$ which is a small perturbation of a surface constructed by gluing together $\mathit{length}(I)/r$ spheres of radius $r$ with centers lying on $\gamma$.

	\item  Let $I$ be a ray of the geodesic $\gamma$. Then there exists an $r_0 > 0$ so that for every $0 < r < r_0$, there is a CMC surface $\Sigma^{\mathit{OE}}_r$ which is a small perturbation of a surface constructed by gluing together a number $O(1/r)$ spheres of radius $r$ with centers lying on $\gamma$, together with an end of a Delaunay surface whose axis lies along $\gamma$.

	\end{itemize}
\end{nonumthm}

\noindent The rescalings of $\Sigma^F_r$ and $\Sigma^{\mathit{OE}}_r$ by the factor $1/r$ converge as $r \to 0$ to an infinite or semi-infinite string of spheres of radius $1$ with centers arranged along a segment or ray in $\R^n$.  The precise mode of convergence will become clear in the course of the proof. 

As already noted, the proof roughly follows the proofs of analogous gluing theorems for CMC surfaces in Euclidean space: an approximately CMC surface $\Sigma$ is deformed via small normal deformations which are parametrized by functions  on $\Sigma$, which transforms the problem to one of finding a solution of the constant mean curvature PDE. One difficulty is the fact that one expects to find a solution only when all parameters are very small, but this means that the geometry of $\Sigma$, and hence the PDE which must be solved, are very degenerate. In addition, the Jacobi operator (i.e.\ the linearized mean curvature operator) on $\Sigma$ has small eigenvalues generated by the nullspaces of the Jacobi operators on each spherical and catenoidal component. The PDE is first solved on the finite codimensional orthogonal complement of this approximate nullspace. By repositioning the various components of this approximate solution one can show that it is possible to kill the remaining finite dimensional piece too, provided the map carrying the `repositioning parameters' to the right hand side of the flux formula \eqref{eqn:balancingtwo} is an isomorphism. One must also keep careful track of the dependence on the neck-sizes $\eps$ and radii $r$ in all of this to guarantee that the estimates controlling the existence of the solution of the constant mean curvature PDE are strong enough. 

\bigskip
\noindent \scshape Acknowledgement. \upshape The authors wish to thank Frank Pacard for interesting discussions during the course of this work.

\section{Preliminary Geometric Calculations}
\label{sec:prelimcalc}

The approximate solutions constructed in Section \ref{sec:approxsol} are assembled from small geodesic spheres  centered on points of the geodesic $\gamma$ in $M$ connected to one another by small catenoidal necks. It is most convenient to use geodesic normal coordinates centered at points of $\gamma$. Since the ambient metric  $g$ is a second order perturbation of the Euclidean metric in these coordinates, the first step in every estimate is to perform the computations for a Euclidean metric; the second step is to incorporate the 
perturbations coming from the metric into the estimates.  In this section we derive various expansions 
of the mean curvature and other geometric quantities. 

\subsection{Geometry of Surfaces in a Geodesic Normal Coordinate Chart}
\label{sec:gnormcoord}

If $p$ is any point in a Riemannian manifold $(M, g)$, then in terms of geodesic normal coordinates centered at $p$, 
$$ 
g := \mathring g + P :=  \big( \delta_{ij} + P_{ij}(x) \big) \, \dif x^i \otimes \dif x^j 
$$
where $\mathring g$ is the Euclidean metric and $P$ is the perturbation term.  It is well known that
\begin{equation}
	\label{eqn:metricexp}
	P_{ij}(x) = \frac{1}{3} \sum_{l,m} R_{iljm}(0) x^l x^m + 
\frac{1}{6} \sum_{l,m,n} R_{iljm;n}(0) x^l x^m x^n  + \mathcal O (\| x \|^4).
\end{equation}
where $R_{ijkl} := \mathrm{Rm}( \frac{\partial}{\partial x^i}, \frac{\partial}{\partial x^j}, \frac{\partial}{\partial x^k}, \frac{\partial}{\partial x^l})$ and $R_{ijkl;m} := \bar \nabla_{\! \! \frac{\partial}{\partial x^m}, } \mathrm{Rm}( \frac{\partial}{\partial x^i}, \frac{\partial}{\partial x^j}, \frac{\partial}{\partial x^k}, \frac{\partial}{\partial x^l})$ are components of the ambient Riemann curvature tensor and its covariant derivative (the ambient covariant derivative is denoted $\bar \nabla$).

Suppose that  $\Sigma$ is a surface in $M$. The following results provide expansions for various geometric quantities of $\Sigma$ in terms of $P$. Here and in the rest of the paper, let $h, \Gamma, \nabla, \Delta, N, 
B, H$ be the induced metric, Christoffel symbols, covariant derivative, Laplacian, unit normal vector, second 
fundamental form, and mean curvature of $\Sigma$ with respect to the metric $g$, and let $\mathring h, \mathring 
\Gamma, \mathring \nabla, \mathring \Delta, \mathring N, \mathring B, \mathring H$ be these same objects with 
respect to the Euclidean metric. Near a point $x \in \Sigma$, let $\{ E_1, E_2\}$ be a local
frame for $T\Sigma$ induced by some coordinate system and denote by $Y := \sum_j x^j \frac{\partial}{\partial x^j}$ 
the position vector.  Define
\begin{align*}
	\mathcal P_{00} &:= P(\mathring N, \mathring N) \\
	\mathcal P_{0j} &:= P(\mathring N, E_j) \\
	\mathcal P_{ij} &:= P(E_i, E_j) \\
	\mathcal P_{ijt} &:= \tfrac{1}{2} \big( (E_i P) (E_j, E_t) + (E_j P)(E_i, E_t) - (E_t P) (E_i, E_j) \big) \\
	\mathcal P_{ij0} &:= \tfrac{1}{2} \big( (E_i P) (E_j, \mathring N) + (E_j P)(E_i, \mathring N) - (\mathring N P) (E_i, E_j) \big) \, .
\end{align*}
Straightforward geometric calculations now yield the following results.

\begin{lemma}
	\label{prop:expansions}
	The induced metric of $\Sigma$ and the associated Christoffel symbols are given by
	$$h_{ij} = \mathring h_{ij} + \mathcal P_{ij} \qquad \mbox{and} \qquad \Gamma_{ijk} = \mathring \Gamma_{ijk} + \mathcal P_{ijk} + \mathcal P_{0k} \mathring B_{ij} \, .$$
	The normal vector of $\Sigma$ satisfies
	\begin{align*}
		N = \frac{ \mathring N - h^{ij} \mathcal P_{0j} E_i }{ \big( 1 + \mathcal P_{00} - \mathring h^{ij} \mathcal P_{0i} \mathcal P_{0j} \big)^{1/2} }  \, .
	\end{align*}
	The second fundamental form of $\Sigma$ satisfies
	\begin{align*}
		B_{ij} &= \big( 1 + \mathcal P_{00} - \mathring h^{ij} \mathcal P_{0i} \mathcal P_{0j} \big)^{1/2}  \mathring B_{ij} + \frac{\mathcal P_{ij0} - \mathring h^{kl} \mathcal P_{0k} \mathcal P_{ijl} }{  \big( 1 + \mathcal P_{00} - \mathring h^{ij} \mathcal P_{0i} \mathcal P_{0j} \big)^{1/2} } \, .
	\end{align*}
\end{lemma}

\begin{lemma}
	\label{prop:approxexpansions}
	The induced metric of $\Sigma$ satisfies
	$$h^{ij} =  \mathring h^{ij} - \mathring h^{is} \mathring h^{jt} \mathcal P_{st} + \mathcal O(\| Y \|^4) \, . $$ 
	The normal vector of $\Sigma$ satisfies
	\begin{align*}
		N = \big( 1 - \tfrac{1}{2} \mathcal P_{00} \big) \mathring N - \mathring h^{ij} \mathcal P_{0j} E_i + \mathcal O(\| Y \|^4)   \, .
	\end{align*}
	The second fundamental form and mean curvature of $\Sigma$ satisfy
	\begin{align*}
		B_{ij} &= \big( 1 + \tfrac{1}{2} \mathcal P_{00} \big)\mathring B_{ij} + \mathcal P_{ij0}  + \mathcal B_{ij} (Y, \mathring B, \mathring N, E_1, E_2)  \\
		H &=\big( 1 + \tfrac{1}{2} \mathcal P_{00} \big)  \mathring H + \mathring h^{ij} \mathcal P_{ij0}  - \mathring B^{ij} \mathcal P_{ij} + \mathcal H(Y, \mathring B, \mathring N, E_1, E_2) \, .
	\end{align*}
where $\mathcal B_{ij}$ and $\mathcal H$ are functions satisfying
$$\max_{i,j} \mathcal | B_{ij} (Y, \mathring B, \mathring N, E_1, E_2) | + | \mathcal H (Y, \mathring B, \mathring N, E_1, E_2) | \leq  C \| Y \|^3 (1 + \| Y \| \| \mathring B \| )$$
for a constant $C$ depending only on the curvature tensor of the ambient manifold at the center of the normal coordinate chart under consideration.
\end{lemma}

The mean curvature of $\Sigma$ may now be computed as
\begin{equation}
	\label{eqn:geomcexp}
	\begin{aligned}
		H &= \big( 1 + \tfrac{1}{6} \riem (\mathring N, Y, \mathring N, Y) + \tfrac{1}{12} \bar \nabla_Y \riem (\mathring N, Y, \mathring N, Y) \big) \mathring H \\
		&\qquad - \big( \tfrac{1}{3} \riem (E_i, Y, E_j, Y) + \tfrac{1}{6} \bar \nabla_Y \riem (E_i, Y, E_j, Y) \big) \mathring B^{ij} \\
		&\qquad - \tfrac{2}{3} \ric (Y, \mathring N) - \tfrac{1}{2} \bar \nabla_Y \ric(Y, \mathring N) + \tfrac{1}{12} \bar \nabla_{\mathring N} \ric(Y, Y) - \tfrac{1}{6} \bar \nabla_{\mathring N} \riem ( \mathring N, Y, \mathring N, Y) \\
		&\qquad + \mathcal O \big( \| Y \|^3 (1 +  \| Y \| \| \mathring B \|) \big)  \, .
	\end{aligned}
\end{equation}
The third line here contains the largest terms coming from the ambient curvature. 

\subsection{Mean Curvature Calculations in Euclidean Space}
\label{sec:euclideancalcs}

Let $\Sigma$ be a surface in Euclidean space.  Choose a function $f: \Sigma \rightarrow \R$ and define $\mathring \mu_f : \Sigma \rightarrow \R^{n+1}$ to be the normal deformation of $\Sigma$ by $f(p)$. The mean curvature operator $f \mapsto 
\mathring H \big[ \mathring \mu_f(\Sigma) \big]$ with respect to the Euclidean metric decomposes as 
\begin{align}
	\label{eqn:eucquad}
	\mathring H \big[ \mathring \mu_f(\Sigma) \big] &= \mathring H  + \mathring{\mathcal L} (f) + \mathring{\mathcal Q} (f) 
\end{align}
where $\mathring{\mathcal L} (f) := \mathring \Delta f + \| \mathring B \|^2 f$ is the linearized mean curvature operator with respect to the Euclidean metric and $\mathring{\mathcal Q} (f)$ is the quadratic and higher remainder term. The second fundamental form can be similarly expanded as $\mathring B \big[ \mathring \mu_f(\Sigma) \big] = \mathring B + \mathring B^{(1)} (f) +  \mathring B^{(2)} (f)$.  We now derive expansions for $ \mathring B^{(1)} (f)$, $ \mathring B^{(2)} (f)$ and $ \mathring{\mathcal Q}(f)$ in terms of $f$.  Although these results are fairly standard, it is important is to track the dependence 
on $\| \mathring B \|$ in the various error terms appearing in the expansions.

We first expand $\mathring B_f := \mathring B \big[ \mathring \mu_f(\Sigma) \big]$ and $\mathring H_f := 
\mathring H \big[ \mathring \mu_f(\Sigma) \big]$ in terms of $f$ and extract the constant, 
linear and higher-order parts.  Introduce
\begin{equation*}
	\beta_{st} := \mathring h_{st} - f \mathring B_{st} \qquad 
	\beta^{st} := \big[\mbox{Inverse of $\beta$}\big]^{st} \qquad
	D := \big( 1 + \beta^{ik} \beta^{jl} f_{,k} f_{,l} \mathring h_{ij} \big)^{1/2}
\end{equation*}
where a comma denotes ordinary differentiation in the coordinate directions. After some work, one finds that the Euclidean induced metric $\mathring h_f$, its inverse, and the Euclidean unit normal vector $\mathring N_f$ of $\mathring \mu_f(\Sigma)$ are
\begin{align*}
	[\mathring h_f]_{ij} &:=  \beta_{ik} \beta_{jl} \mathring h^{kl} + f_{,i} f_{,j} \\
	[\mathring h_f^{-1}]^{ij} &:= \beta^{im} \beta^{jn} \left( \mathring h_{mn} - \frac{ \beta^{kp} \beta^{lq} f_{,k} f_{,l} \mathring h_{mp} \mathring h_{nq}}{D^2} \right) \\
	\mathring N_f &:= \frac{1}{D} \big( \mathring N - \beta^{ij} f_{,i} E_j \big).
\end{align*}
The second fundamental form $\mathring B_f := \mathring B[\mathring \mu_f(\Sigma)]$ can thus be expressed in terms of $f$ as
\begin{equation}
	\label{eqn:quadsecondff}
	\begin{aligned}[]
		[\mathring B_f]_{ij} &= D^{-1}\left(\mathring B_{ij} + f_{;ij} - f \mathring B^k_i \mathring B_{kj} + \beta^{kl} f_{,l} ( f_{,i} \mathring B_{jk} +  f_{,j} \mathring B_{ik} + f  \mathring B_{jk;i}  )\right)
	\end{aligned}
\end{equation}
where a semicolon denotes the covariant derivative of $\Sigma$ with respect to $\mathring h_{st}$. We expand the inverse of the induced metric as 
$$\mathring h_f^{ij} := \mathring h^{ij} + 2 f  \mathring B^{ij} - \frac{1}{D^2} \beta^{im} \beta^{jn} \beta^{kp} \beta^{lq} f_{,k} f_{,l} \mathring h_{mp} \mathring h_{nq} + \eta^{im} \eta^{jn} \mathring h_{mn}$$
where the remainder $\eta^{ij}$ in $\beta^{ij} := \mathring h^{ij} + f \mathring B^{ij} + \eta^{ij}$ satisfies 
$| \eta^{ij} | = \mathcal O(|f|^2 \| \mathring B \|^2)$.  Now taking the trace of \eqref{eqn:quadsecondff} with respect to $\mathring h_f^{ij} $ yields the mean curvature $\mathring H_f := \mathring H(\mathring \mu_f(\Sigma)$ which is
\begin{equation}
	\label{eqn:quadmeancurv}
	\begin{aligned}
		D \times \mathring H_f &= \mathring H +  \mathring \Delta f +\| \mathring B \|^2 f + 2 f \mathring B^{ij} f_{;ij}   - 2 f^2  \mathrm{Tr}_\delta (\mathring B^3) + \mathring h^{ij} \beta^{kl} f_{,l} (2 f_{,i} \mathring B_{jk}  + f  \mathring B_{jk;i}  )  \\
		& + \left( - \frac{1}{D^2} \beta^{im} \beta^{jn}  \beta^{kp} \beta^{lq} f_{,k} f_{,l}  \mathring h_{mp} \mathring h_{nq} + \eta^{im} \eta^{jn} \mathring h_{mn} \right) ( \mathring B_{ij} + f_{;ij} - f \mathring B^k_i \mathring B_{kj} ) \\
		& + \left( 2 f \mathring B^{ij}   - \frac{1}{D^2} \beta^{im} \beta^{jn}  \beta^{kp} \beta^{lq} f_{,k} f_{,l}  \mathring h_{mp} \mathring h_{nq} + \eta^{im} \eta^{jn} \mathring h_{mn} \right) \beta^{kl} f_{,l} (2 f_{,i} \mathring B_{jk}  + f  \mathring B_{jk;i}  ) \, .
	\end{aligned}
\end{equation}
All of this is summarized in the following lemma.

\begin{lemma}
The linear parts of $\mathring B_f$ and $\mathring H_f$ are 
\begin{align*}
	[\mathring B^{(1)}(f)]_{ij} &:= f_{;ij} - f \mathring B_i^k \mathring B_{kj} \\
	\mathring{\mathcal L} (f) &:= \mathring \Delta f + \| \mathring B \|^2 f \, .
\end{align*}
The quadratic remainder parts of $\mathring B_f$ and $\mathring H_f$ are
\begin{align*}
	[\mathring B^{(2)}(f)]_{ij} &:= \frac{\beta^{kl} f_{,l}}{D} \big(   f_{,i} \mathring B_{jk} +  f_{,j} \mathring B_{ik} + f  \mathring B_{jk;i}   \big)  +  \left( \frac{1}{D} - 1 \right) \big( \mathring B_{ij} + f_{;ij} - f \mathring B_i^k \mathring B_{kj} \big)\\
	\mathring{\mathcal Q} (f) &:= \left. \frac{1}{D} \right[ 2 f \mathring B^{ij} f_{;ij}   - 2 f^2  \mathrm{Tr}_\delta (\mathring B^3) + \mathring h^{ij} \beta^{kl} f_{,l} (2 f_{,i} \mathring B_{jk}  + f  \mathring B_{jk;i}  ) \\
	&  \qquad \; + \left( - \frac{1}{D^2} \beta^{im} \beta^{jn}  \beta^{kp} \beta^{lq} f_{,k} f_{,l}  \mathring h_{mp} \mathring h_{nq} + \eta^{im} \eta^{jn} \mathring h_{mn} \right) ( \mathring B_{ij} + f_{;ij} - f \mathring B^k_i \mathring B_{kj} ) \\
	& \qquad \; + \left. \left( 2 f \mathring B^{ij}   - \frac{1}{D^2} \beta^{im} \beta^{jn}  \beta^{kp} \beta^{lq} f_{,k} f_{,l}  \mathring h_{mp} \mathring h_{nq} + \eta^{im} \eta^{jn} \mathring h_{mn} \right) \beta^{kl} f_{,l} (2 f_{,i} \mathring B_{jk}  + f  \mathring B_{jk;i}  ) \right] \\
	& \qquad \!\!+  \left( \frac{1}{D} - 1 \right) \big(  \mathring H +  \mathring \Delta f +\| \mathring B \|^2 f \big) \, .
\end{align*}
\end{lemma}

The quadratic parts of both $\mathring B_f$ and $\mathring H_f$ are unwieldy, but only basic structural facts about
them are needed in the sequel. To simplify matters, we suppose that $| f| \| \mathring B \| + \| \mathring 
\nabla f \| \ll 1$, which will be justified later on.  Furthermore, both $\mathring B^{(2)}(f)$ and 
$\mathring{\mathcal Q}(f)$ can be expanded into a sum of terms which are each linear combinations 
of the coefficients of the tensor  $f^i \cdot (\mathring \nabla f)^{\otimes j} \otimes (\mathring \nabla^2 f 
)^{\otimes k} \otimes \mathring B^{\otimes l} \otimes (\mathring \nabla \mathring B)^{\otimes m}$, where 
$i, j, k, l, m$ are positive integers such that $k+m \leq 1$ and $i + 1 = k + l + 2 m$ (i.e.\ the number 
of times the function $f$ appears is smaller by one than the sum of the number of covariant derivatives 
and the number of occurrences of the second fundamental form).  Consequently the dominant terms in 
$\mathring B^{(2)}(f)$ and $\mathring{\mathcal Q}(f)$ are $\mathcal O(1)$ linear combinations of components of
\begin{equation*}
(\mathring \nabla f )^2 \! \otimes \! \mathring B, \quad 
f\, \mathring \nabla f \! \otimes \! \mathring B^{2}, \quad 
f\, \mathring \nabla f \! \otimes \! \mathring \nabla \mathring B, 
\quad f^2 \mathring B^3,  \quad f \, \mathring \nabla^2 f \! \otimes \! \mathring B
\quad \mbox{and} \quad (\mathring \nabla f)^2 \! \otimes \! \mathring \nabla^2 f \, .
\end{equation*}
The following estimates are now straightforward consequences of this discussion. 

\begin{lemma}
	\label{lemma:quadeuclexp}
	Assuming that $| f_i | \| \mathring B \| + \| \mathring \nabla f_i  \| \ll 1$ for $i = 1, 2$, the 
quadratic remainders in the second fundamental form and mean curvature satisfy 
	\begin{align*}
		| \mathring B^{(2)}(f_1) - \mathring B^{(2)}(f_2) | + |\mathring {\mathcal Q}(f_1) - \mathring {\mathcal Q}(f_2) | & \\
		&\hspace{-25ex} \leq C | f_1 - f_2 | \cdot \max_i \big(|f_i| \| \mathring B \|^3 + \|\mathring \nabla f_i \| \| \mathring B \|^2 +  \|\mathring \nabla f_i \| \| \mathring \nabla \mathring B \|  + \| \mathring \nabla^2 f_i \| \| \mathring B \| \big) \\
		&\hspace{-25ex} \qquad + C \| \mathring \nabla f_1 - \mathring \nabla f_2 \| \cdot \max_i \big(  | f_i | \| \mathring B \|^2 + \|\mathring \nabla f_i \| \| \mathring B \| + | f_i | \| \mathring \nabla \mathring B \| + \| \mathring \nabla^2 f_i \|  \big) \\
		&\hspace{-25ex} \qquad + C \| \mathring \nabla^2 f_1 - \mathring \nabla^2 f_2 \| \cdot \max_i \big( |f_i| \| \mathring B \| + \| \mathring \nabla f_i \| \big) \\
		&\hspace{-25ex} \qquad + C \| \mathring \nabla^2 f_1 - \mathring \nabla^2 f_2 \| \cdot \max_i \| \mathring \nabla f_i \|^2 + C \| \mathring \nabla f_1 - \mathring \nabla f_2 \| \cdot \max_i \| \mathring \nabla f_i \| \| \mathring \nabla^2 f_i \|
	\end{align*}
	where $C$ is independent of $f_1,f_2$ and $\| \mathring B \|$.
\end{lemma}

\subsection{Mean Curvature Calculations for a Perturbed Background Metric}
\label{sec:perturbedcalcs}

Consider now a surface $\Sigma_f := \mathring \mu_f (\Sigma)$ deformed by the amount $f$ in the direction of
the Euclidean normal to $\Sigma$. Working again in a geodesic normal coordinate system centered at some
point of $\gamma$, we now decompose the mean curvature operator $f \mapsto H_f := H [\Sigma_f]$ as
\begin{equation}
	\label{eqn:meancurvfnexp}
	H_f = H + \mathcal L(f) + \mathcal Q(f) + \mathcal H(f)
\end{equation}
into constant, linear and quadratic remainder parts plus a small error term.

The key is to substitute the tangent vector fields $[E_f]_ {i} := \mathring h^{jk} \beta_{ij} E_k +  f_{,i} \mathring N $ of $\mathring \mu_{f} (\Sigma)$ (here $E_1, E_2$ are tangent vectors for $\Sigma$),  the Euclidean normal vector field $ \mathring N_{f} := \frac{1}{D} \big(  \mathring N - \beta^{kj} f_{,k}  E_j \big)$ of $\mathring \mu_{f} (\Sigma)$, and the position vector field $Y_ {f} := Y +  f \mathring N$ of $\mathring \mu_{f} (\Sigma)$ relative to the center of the normal coordinate chart, as well as the expressions for $\mathring H$ and $\mathring B$ into the formul\ae\ from Lemma \ref{prop:approxexpansions}.  
This yields
\begin{equation}
	\label{eqn:meancurvf}
	H = \big( 1 + R_1 (f) \big) \mathring H +  [R_2(f) ]_{ij} \mathring B^{ij} + \bar R( f) + \mathcal H(f)  \, ,
\end{equation}
where $R_1$, $R_2$ and $\bar R$ are first-order differential operators and $\mathcal H(f) := \mathcal H ( Y_f, \mathring B_f, \mathring N_f, [E_f]_1, [E_f]_2 )$. As before, the precise structure of these operators is not important, though we
still must estimate their dependence on $f$, $Y$ and $\mathring B$. 

First, by examining the expansions for $R_s$ in terms of $f$, $Y$ and $\mathring \nabla f$, and for $D$ and $\beta$ in terms of $f \mathring B$ and $\mathring \nabla f$, one deduces that $R_s$ has an expansion into constant, linear and quadratic remainder terms of the form $R_s(f) :=  R_s^{(0)} + R_s^{(1)}(f) + R_s^{(2)}(f)$  where
\begin{equation}
	\label{eqn:firstcurvexp}
	\begin{aligned}
		R_1^{(0)}(f) &:= \tfrac{1}{6} \riem (\mathring N, Y, \mathring N, Y) \\[0.5ex]
		R_1^{(1)}(f) &:= -  \tfrac{1}{3} \riem(\mathring N, Y, E_j , Y) \mathring h^{jk} f_{,k}   \\[0.5ex]
		[R_2^{(0)}(f)]_{ij}  &:= \tfrac{1}{6} \riem (E_i, Y, E_j, Y)   \\[0.5ex]
		R_2^{(1)}(f) &:= \tfrac{1}{3} \big( \riem (E_i, Y, E_j, \mathring N) - \riem (E_k, Y, E_i, Y) \mathring B^k_j \big) f + \tfrac{1}{3}  \riem (\mathring N, Y, E_i , Y) f_{,j} \\
		&\qquad  + \big(  \tfrac{1}{3}  \bar \nabla_Y \riem (E_i, Y, E_j, \mathring N)   + \tfrac{1}{12} \bar\nabla_{\mathring N} \riem (E_i, Y, E_j, Y) \big) f   \\[0.5ex]
		R_s^{(2)}(f) &:=  R^{(2)}_{s,0} (f, \mathring \nabla f, Y) \cdot R^{(2)}_{s,1} (f \mathring B, \mathring \nabla f)
	\end{aligned}
\end{equation}
for $s = 1, 2$.  Here $R^{(2)}_{s,0}$ is a sum of quadratic and higher expressions in the components of $f^k \, Y^{\otimes (n-k)} \otimes (\mathring \nabla f)^{\otimes l}$ for various $n \geq 2$, $k \leq n$ and $l \geq 0$ whose coefficients are bounded by curvature quantities, while $R^{(2)}_{s,1}$ can be expanded to any order in a power series in the components of $f \mathring B$ and $\mathring \nabla f$ whose coefficients are bounded by curvature quantities.   One finds a similar expansion for $\bar R (f )$ into constant, linear and quadratic remainder terms of the form $\bar R (f ) := \bar R^{(0)} + \bar R^{(1)}(f) + \bar R^{(2)}(f)$ where
\begin{equation}
	\label{eqn:secondcurvexp}
	\begin{aligned}
		\bar R^{(0)}(f) &:= - \tfrac{2}{3} \ric (Y, \mathring N) - \tfrac{1}{2} \bar \nabla_Y \ric(Y, \mathring N) + \tfrac{1}{12} \bar \nabla_{\mathring N} \ric(Y, Y)  - \tfrac{1}{6} \bar \nabla_{\mathring N} \riem ( \mathring N, Y, \mathring N, Y)  \\[0.5ex]
		\bar R^{(1)}(f) &:= - \big( \tfrac{2}{3} \ric (\mathring N, \mathring N) + \tfrac{1}{2}  \bar \nabla_Y \ric(\mathring N, \mathring N) + \tfrac{1}{2}  \bar \nabla_{\mathring N} \ric(Y, \mathring N)  - \tfrac{1}{6} \bar \nabla_{\mathring N} \ric(\mathring N, Y) \big) f \\
		&\qquad + \big( \tfrac{2}{3} \ric (Y, E_i) + \tfrac{1}{2}  \bar \nabla_Y \ric(Y,E_i) - \tfrac{1}{12} \bar \nabla_{E_i} \ric(Y, Y)\big)  \mathring h^{ij} f_{,j}   \\
		&\qquad + \big( \tfrac{1}{3} \bar \nabla_{\mathring N} \riem ( E_i , Y, \mathring N, Y) + \tfrac{1}{6} \bar \nabla_{E_i } \riem ( \mathring N, Y, \mathring N, Y) \big) \mathring h^{ij} f_{,j} \\[0.5ex]
		\bar R^{(2)} (f) &:= \bar R^{(2)}_0 (f,  \mathring \nabla f, Y) \cdot \bar R^{(2)}_1 (f \mathring B, \mathring \nabla f)
	\end{aligned}
\end{equation}
for $s=1,2$.  Here $\bar R^{(2)}_0 $  is a sum of quadratic and higher expressions in the components of $f^k \, Y^{\otimes (n-k)} \otimes (\mathring \nabla f)^{\otimes l}$ for various $n \geq 1$, $k \leq n$ and $l \geq 0$ whose coefficients are bounded by curvature quantities, while $\bar R^{(2)}_1$ can be expanded to any order in a power series in the components of $f \mathring B$ and $\mathring \nabla f$ whose coefficients are bounded by curvature quantities.   Therefore the following estimates hold.

\begin{lemma}
	\label{lemma:curvtermest}
The constant and linear parts of $R_s (f )$ for $s=1, 2$ and $\bar R (f )$ satisfy
\begin{align*}
	| R_s^{(0)}(f) | &\leq C \| Y \|^2 \\
	| R_s^{(1)}(f) | &\leq C  \| Y \|  \big( |f | + \|Y\| \| \mathring \nabla f \| \big) \\
	| \bar R^{(0)}(f) | &\leq C \| Y \| \\
	| \bar R^{(1)}(f) | &\leq C  \big( |f | + \|Y\| \| \mathring \nabla f \| \big) \, .
\end{align*}
The quadratic and higher parts of $R_s (f )$ for $s=1, 2$ and $\bar R (f )$ satisfy 
\begin{align*}
	| R_s^{(2)}(f_1) - R_s^{(2)}(f_2) | &\leq C \big(  | f_1 - f_2 |+ \| Y \| \| \mathring \nabla f_1 - \mathring \nabla f_2 \| \big) \cdot \max_{i} \big( | f_i |+ \| Y \| \| \mathring \nabla f_i \|\big) \\[1ex]
	| \bar{{R}}^{(2)}(f_1) - \bar{{R}}^{(2)}(f_2) | &\leq  C | f_1 - f_2| \cdot \max_i \| \mathring \nabla f_i \| + C \| \mathring \nabla f_1 - \mathring \nabla f_2 \| \cdot \max_i | f_i | \\[-0.5ex]
	&\qquad + C \big(  | f_1 - f_2 |+ \| Y \| \| \mathring \nabla f_1 - \mathring \nabla f_2 \| \big) \cdot \max_{i} \big( | f_i |+ \| Y \| \| \mathring \nabla f_i \|\big) \, .
\end{align*}
In these estimates, $C$ is a constant depending only on the curvature tensor of the ambient manifold at the center of the normal coordinate chart under consideration.
\end{lemma}

One now substitutes the expansions for $R(f)$ and $\bar R(f)$ along with the expansions for $\mathring H_{f}$ and $\mathring B_{f}$ in terms of $f$ into equation \eqref{eqn:meancurvf} and extracts the various parts.  That is, by performing these substitutions, one finds
\begin{equation}
	\label{eqn:pertquadmc}
	\begin{aligned}
		H &:= ( 1 + R_1^{(0)}) \mathring H + \bar R^{(0)} + R_2^{(0)} \cdot \mathring B^{(0)}\\
		\mathcal L (f) &:= ( 1 + R_1^{(0)}) \mathring {\mathcal L} (f) + \bar R^{(1)}(f) + R^{(0)} \cdot \mathring B^{(1)}(f) + R^{(1)}(f) \cdot \mathring B  \\
		\mathcal Q (f) &:= (1 + R(f)) \mathring {\mathcal Q} (f) + \bar{{R}}^{(2)} (f) +  R^{(1)}(f) \cdot \mathring B^{(1)}(f)  + R^{(1)}(f) \mathring {\mathcal L} (f) + R(f) \cdot \mathring B^{(2)}(f) 
	\end{aligned}
\end{equation} 
where $[R(f)]_{ij} := R_1(f) \mathring h_{ij} + [R_2(f)]_{ij}$. Moreover, the following estimates hold.

\begin{lemma}
	\label{lemma:newfnlinest}
	The operator $\mathcal L$ satisfies
	\begin{equation}
	\label{eqn:lindifest}
	\begin{aligned}
		| \mathcal L(u) - \mathring{\mathcal L}(u) | &\leq C (1 + \| Y \| \| \mathring B \|) (|u| + \|Y \|  \| \mathring \nabla u \| + \| Y \|^2 \| \mathring \nabla^2 u \|)	\end{aligned}  
\end{equation}
where $C$ is a constant depending only on the curvature tensor of the ambient manifold at the center of the normal coordinate chart under consideration.
\end{lemma}

\begin{lemma}
	\label{lemma:newfnquadest}
	Under the assumption that $| f_i| \| \mathring B \| + \| \mathring \nabla f_i \| \ll 1$ for $i = 1, 2$, the quadratic remainder $\mathcal Q (f_i)$ satisfies
\begin{align*}
		 | \mathcal Q(f_1) -\mathcal Q(f_2) | & \leq C | f_1 - f_2 | \cdot \max_i \big(|f_i| \| \mathring B \|^3 + \|\mathring \nabla f_i \| \| \mathring B \|^2 +  \|\mathring \nabla f_i \| \| \mathring \nabla \mathring B \|  + \| \mathring \nabla^2 f_i \| \| \mathring B \| \big) \\
		&\qquad + C \| \mathring \nabla f_1 - \mathring \nabla f_2 \| \cdot \max_i \big(  | f_i | \| \mathring B \|^2 + \|\mathring \nabla f_i \| \| \mathring B \| + | f_i | \| \mathring \nabla \mathring B \| + \| \mathring \nabla^2 f_i \|  \big) \\
		&\qquad + C \| \mathring \nabla^2 f_1 - \mathring \nabla^2 f_2 \| \cdot \max_i \big( |f_i| \| \mathring B \| + \| \mathring \nabla f_i \| \big) \\
		&\qquad + C \| \mathring \nabla^2 f_1 - \mathring \nabla^2 f_2 \| \cdot \max_i \| \mathring \nabla f_i \|^2 + C \| \mathring \nabla f_1 - \mathring \nabla f_2 \| \cdot \max_i \| \mathring \nabla f_1 \| \| \mathring \nabla^2 f_i \| \\
		&\qquad +C | f_1 - f_2| \cdot \max_i \| \mathring \nabla f_i \| + C \| \mathring \nabla f_1 - \mathring \nabla f_2 \| \cdot \max_i | f_i | \\
		&\qquad + C \big(  | f_1 - f_2 |+ \| Y \| \| \mathring \nabla f_1 - \mathring \nabla f_2 \| \big) \cdot \max_{i} \big( | f_i |+ \| Y \| \| \mathring \nabla f_i \|\big) \, .
\end{align*}
In these estimates, $C$ is a constant depending only on the curvature tensor of the ambient manifold at the center of the normal coordinate chart under consideration.
\end{lemma}

\noindent Finally, further straightforward calculation leads to the remaining estimate for $\mathcal H$. 

\begin{lemma}
	\label{lemma:newfnerrorest}
	Under the assumption that $| f| \| \mathring B \| + \| \mathring \nabla f \| \ll 1$ then the error term $\mathcal H(f)$ satisfies
	$$| \mathcal H(f_1) -   \mathcal H(f_2)  | \leq C  \| Y \|^2  (1 + \| Y \| \| \mathring B \|)  \big(|f_1 - f_2| +  \|Y \| \| \mathring \nabla f_1 - \mathring \nabla f_2 \| + \|Y \|^2 \| \mathring \nabla^2 f_1 - \mathring \nabla^2 f_2 \| \big)$$
	where $C$ is a constant depending only on the curvature tensor of the ambient manifold at the center of the normal coordinate chart under consideration.
\end{lemma}

\section{The Approximate Solutions}
\label{sec:approxsol}
We now construct two families of approximate CMC surfaces.  The first family consists of finite-length surfaces 
invariant under the reflection $t \mapsto -t$ constructed by gluing together $K$ small geodesic spheres of radius 
$r$ along the $(t, 0, 0)$ geodesic with small interpolating necks. Here $K$ is approximately $1/r$ so that 
the surface fills out a region along $\gamma$ of bounded length which does not tend to $0$ with $r$.  
The second family consists of one-ended surfaces in an asymptotically flat Riemannian manifold.  These are 
constructed by taking a configuration of $K$ spheres as above and then attaching a semi-infinite 
Delaunay surface to the last sphere. These two families are denoted $\apsolf$ and $\apsoloe$, respectively.
These depend on parameters $\sigma_1, \sigma_2, \ldots$ and $\delta_1, \delta_2, \ldots$ which govern the 
precise locations of the component spheres and necks. For brevity, we often just write $\apsol$ for either
family when the context is clear or does not matter. 

\begin{rmk} Starting from this point, our presentation will be phrased in terms of two-dimensional surfaces $\Sigma$ contained in a three-dimensional Riemannian manifold $M$.  This is done for the purpose of simplicity; however, everything that follows can be easily adapted to the $(n+1)$-dimensional setting.  
\end{rmk}

\subsection{The Finite-Length Surface} 

Let $\gamma$ be the $(t, 0, 0)$ geodesic and (with slight abuse of notation) also the arc-length parametrization of this geodesic given by $\gamma(t) := (t, 0, 0)$.  Introduce the small radius $r$ and the (much smaller) separation parameters $\sigma_{k}$.  Let $t_0 = 0$ and $t_k := 2 k r + \sum_{l=0}^{k-1} (\sigma_l )$ and let $p_{\pm k} := (\pm t_k, 0, 0)$.  The geodesic spheres that will be used in the gluing construction are $S_{\pm k} := \partial B_r( p_{\pm k})$. Also let $p_{k}^\flat := \gamma( t_k + r + \sigma_{k}/2)$ be the point half-way between $S_k$ and $S_{k+1}$ and let $p_{k}^\pm := \gamma( t_{k} \pm r)$ be the points in $S_k \cap \gamma$ .  Define $p_{-k}^\pm$ and $p_{-k}^\flat$ in a symmetrical manner.  In the definitions above, the index $k$ ranges from zero to $K$.  

The construction of the first surface consists of three steps.  The first step is to replace each $S_{k}$ with $\tilde S_k$ which is obtained from $S_k \setminus \{ p_k^+, p_k^-\}$ or $S_0$ by a small normal perturbation designed to make $S_{k}$ look more like a catenoid near $p_k^\pm$.  The next step is to find the truncated and rescaled catenoids that fit optimally into the space between $\tilde S_k$ and its neighbours, the precise location of which is governed by displacement parameters $\delta_k$.  The final step is to use cut-off functions to glue each $\tilde S_k$ smoothly to its neighbouring necks. The result of this process will be a family of surfaces that depends on $r$ and the parameters  $\sigma_1, \sigma_2, \ldots$ and $\delta_1, \delta_2, \ldots$.   A number of additional small parameters $\eps_k^\pm$ and $\eps_k$ will be introduced below and it will be shown how these depend on the $\sigma$ and $\delta$.  Denote by $\eps := \max_k \{ \eps_k, \eps_k^+, \eps_k^- \}$ below and in the rest of the paper.

\paragraph*{Step 1.}  Let ${\mathcal L}_S := \mathring \Delta + \| \mathring B \|^2$ be the linearized mean curvature operator of $S_k$ with respect to the Euclidean metric and let $J_k : S_k \rightarrow \R$ be the smooth function in the kernel of ${\mathcal L}_S$ that is cylindrically symmetric with respect to the axis defined by the geodesic $\gamma$ and normalized to have unit $L^2$-norm.  (It is defined by taking the correct multiple of the Euclidean normal component of the translation vector field $\frac{\partial }{\partial t}$.) 

Now introduce small positive scale parameters $\eps_k^{\, \pm}$ that have yet to be determined and define $G_k : S_k \setminus \{ p_k^\pm \} \rightarrow \R$ for $k = 0, \ldots, K-1$ as the unique solution of the equation
\begin{subequations}
\label{eqn:green}
\begin{equation}
	\label{eqn:greenone}
	{\mathcal L}_S (G_k) = \eps_k^{\, +} \delta_+ +  \eps_k^{\, -} \delta_{-} + A_k J_k
\end{equation}
where $\delta_\pm$ denotes the Dirac $\delta$-function at $p_k^\pm$ and the real number $A_k$ is chosen to ensure that the right hand side of \eqref{eqn:greenone} is $L^2$-orthogonal to $J_k$.  Also, let $G'_0$ be the unique solution of the equation 
\begin{equation}
	\label{eqn:greentwo}
	{\mathcal L}_S (G_K') = \eps_K^{\, -} \delta_{-} + A_K' J_K
\end{equation}
\end{subequations}
where $A_K'$ is chosen to ensure that the right hand side of \eqref{eqn:greentwo} is $L^2$-orthogonal to $J_K$. Note that $|A_k| \leq C \eps$ for $k = 1, \ldots, N$ and $A_0 = 0$ by symmetry.

To complete this step, introduce another small radius parameter $r_k$ yet to be determined, and define $\tilde S_k$ as the Euclidean normal graph over $S_k \setminus \big[ B_{r_k/2}(p_k^+) \cup B_{r_k/2}(p_k^-) \big]$ that is generated by the function $r G_k$.  Also define the terminal sphere $\tilde S_{K}'$ as the Euclidean normal graph over $S_{K} \setminus  B_{r_K/2} (p_{K}^-) $ that is generated by the function $r G_{K}'$ as well as its symmetrical counterpart under the $t \mapsto - t$ symmetry.

\paragraph*{Step 2.} Coordinatize a neighbourhood of $p_k^\flat$ using geodesic normal coordinates centered at $p_k^\flat$ and scaled by a factor of $r$.   Let the coordinate map be $\psi_k : B_{R} (p_k^\flat) \subseteq M \rightarrow (-R',R') \times B_{R'}(0) \subseteq \R \times \R^2$, where $R, R'$ are appropriate radii (one should think of $R = O(r)$ and $R' = \mathcal O(1)$).  Note that the $\R$ coordinate corresponds to a translation of the scaled arc-length coordinate along $\gamma$ and $\gamma$ itself maps to the curve $x^0 \mapsto (x^0,0)$ with $p_k^\flat$ mapping to the origin.  The gluing procedure that will now be described applies to any pair of perturbed spheres $\tilde S_k, \tilde S_{k+1}$, including the last pair $\tilde S_{K-1}, \tilde S_K$. The images of $\tilde S_{k+1}$ and $\tilde S_{k}$ under the coordinate map, at least near the origin, can be represented as graphs over the $\R^2$ factor of the form $ \{ (x^0, x) : x^0 = F_{\mathit{sph}}^\pm(x) \}$ where $+$ and $-$ correspond to $\psi_k(\tilde S_{k+1})$ and $\psi_k(\tilde S_{k})$ respectively.  One can check that the Taylor series expansions for $G_k$ near $p_k^\pm$ and for $\psi_k$ imply corresponding expansions for $F_{\mathit{sph}}^\pm$ near $x =0$, which results in the fact that $\psi_k(\tilde S_{k+1})$ is the set of points
$$x^0  = F_{\mathit{sph}}^-(x) := \frac{\sigma_k}{2 r} + \eps_{k+1}^{\, -}  \big( c_{k+1}^- + C_{k+1}^- \log (\| x \|) \big) + \mathcal O(\|x\|^2) + \mathcal O(\eps \|x\|^2)$$
while $\psi_k(\tilde S_{k+1})$ is the set of points
$$x^0 = F_{\mathit{sph}}^+( x) := -\frac{\sigma_k}{2 r} - \eps_{k}^{\, +} \big( c_{k}^+ + C_{k}^+ \log (\| x \|) \big) + \mathcal O(\|x\|^2) + \mathcal O(\eps \|x\|^2) \, .$$  
Here, $c_k^\pm$ and $C_k^\pm$ are constants.
 
The interpolation between $\psi_k(\tilde S_{k+1})$ and $\psi_k(\tilde S_{k})$, will be done using a standard catenoid that has been scaled by a factor of $\eps_k$ and translated by a small amount along its axis.  The $x^0 > 0$ end of such catenoid is given by
\begin{align*}
	x^0 = F_{\mathit{neck}}^+(\eps_k, d_k ; x) :=&\; \eps_k \arccosh(  \|x \| / \eps_k) + \eps_k d_k \\
	=&\; \eps_k  \big( \log(2) - \log(\eps_k ) \big) + \eps_k \log( \| x \|) + \eps_k d_k + \mathcal O( \eps_k^3 / \|x\|^2)
\end{align*}
near the origin, where $d_k$ is the translation parameter. The $x^0 < 0$ end is given by
\begin{align*}
x^0 = F_{\mathit{neck}}^-(\eps_k, d_k ; x)  :=&\; - \eps_k \arccosh(  \|x \| / \eps_k) + \eps_k d_k \\
	=&\; - \eps_k \big( \log(2) - \log(\eps_k) \big) - \eps_k \log( \| x \|) + \eps_k d_k + \mathcal O( \eps_k^3 / \|x\|^2)
\end{align*}
near the origin.  Optimal matching of these asymptotic expansions of the catenoid with those of $\psi_k(\tilde S_{k+1})$ and $\psi_k(\tilde S_{k})$ given above then requires
\begin{gather}
	\eps_k^{\, +} = \frac{\eps_k}{C_k^+} \qquad \quad
	\eps_{k+1}^{\, -} = \frac{\eps_k}{C_{k+1}^-} \qquad \quad
	d_k = \frac{1}{2} \left( \frac{c_{k+1}^-}{C_{k+1}^-} - \frac{c_{k}^+}{C_{k}^+}\right) \notag \\
	\intertext{and}
	\sigma_k = \Lambda_k(\eps_k) := r \eps_k \left( 2 \big( \log(2) - \log(\eps_k) \big) -  \frac{c_{k+1}^-}{C_{k+1}^-} - \frac{c_{k}^+}{C_{k}^+} \right) \, . 	\label{eqn:neckseprel}
\end{gather}
These equations imply that once the spacing $\sigma_k$ between $\tilde S_{k+1}$ and $\tilde S_{k}$ has been decided upon, then one can determine $\eps_k$ by inverting the equation $\sigma_k = \Lambda_k (\eps_k)$, and then the other parameters describing the optimally matched neck can be computed from $\eps_k$.  

Finally, observe that the matching between $\psi_k(\tilde S_{k+1})$ and $\psi_k(\tilde S_{k})$ and the neck defined by the choice of parameters above is \emph{most} optimal in the region of $x$ where the error quantity $\mathcal O(\| x \|^2) + \mathcal O( \eps_k^3 / \|x\|^2)$ is smallest.  It is easy to check that this occurs when $\|x \| = \mathcal O (\eps_k^{3/4})$.  Hence the radii $r_k$ in Step 1 should be chosen as $r_k := r \eps_k^{3/4}$, where the factor of $r$ takes the scaling into account. 
 
\paragraph*{Step 3.} Choose $d_k$, $\eps_k^\pm$ and $\eps_k$ as above.  Also, introduce displacement parameters $\delta_1, \delta_2, \ldots$ that serve to slightly displace the necks from their optimal locations.  Define a smooth, monotone cut-off function $\chi: [0,\infty) \rightarrow [0,1]$ which equals one in $[0,\frac{1}{2}]$ and vanishes outside $[0,1]$. Define the functions $\tilde F^\pm : B_1(0) \rightarrow \R$ by
\begin{equation}
	\label{eqn:interpfn}
	\tilde F^\pm(x)  := \chi (\| x \| / \eps_k^{3/4}) F_{\mathit{neck}}^\pm (\eps_k, d_k + \delta_k; x) + \big( 1 - \chi(\| x \| / \eps_k^{3/4}) \big) F_{\mathit{sph}}^\pm (x) \, .
\end{equation}
Now define the catenoidal interpolation between  $\psi_k(\tilde S_{k+1})$ and $\psi_k(\tilde S_{k})$ as
$$\tilde N_k := \{ (\tilde F^+(x), x) : \eps_k \leq \|x\| \leq \eps_k^{3/4} \} \cup  \{ (\tilde F^-(x), x) : \eps_k \leq \|x\| \leq \eps_k^{3/4} \} \, .$$
Finally, one can define the finite-length approximate solutions as follows.  

\begin{defn}
	Let $K$ be given.  The finite-length surface with parameters $\sigma := \{ \sigma_1, \ldots, \sigma_K \}$ and $\delta := \{ \delta_1, \ldots, \delta_K\}$ is the surface given by
\begin{align*}
	\apsolf &:=  \Big[ \tilde S_0 \setminus \big[B_R(p_0^+) \cup B_R(p_0^-) \big] \Big] \\
	&\qquad \cup \left[ \bigcup_{k=1}^{K-1} \tilde S_k \setminus \big[B_R(p_k^+) \cup B_R(p_k^-) \big]  \right] \cup \Big[ \tilde S_K \setminus B_R(p_K^-) \Big] \cup \left[ \bigcup_{\substack{k=0}}^{K-1} \psi_k^{-1} (\tilde N_k) \right]  \\
	&\qquad \cup \pi \left( \left[ \bigcup_{k=1}^{K-1} \tilde S_k \setminus \big[B_R(p_k^+) \cup B_R(p_k^-) \big]  \right] \cup \Big[ \tilde S_K \setminus B_R(p_K^-) \Big]\cup \left[ \bigcup_{\substack{k=0}}^{K-1} \psi_k^{-1} (\tilde N_k) \right] \right) \, .
\end{align*}
where $\pi$ is the $t \mapsto -t$ reflection.
\end{defn}

\subsection{The One-Ended Surface}
The one-ended family of approximate solutions is constructed by attaching a half Delaunay surface to a finite-length 
surface as constructed above. This Delaunay surface has very small necksize parameter, so we will need the
detailed analysis of these from \cite{mazzeopacardends}. 

\paragraph*{Step 1.} Let $K$ be a large integer.  Exactly as in the previous section, construct $K$ perturbed spheres of the form $ \tilde {S}_k$ for $k = 1, \ldots, K$ and one terminal perturbed sphere of the form $\tilde {S}_0'$ (which in this case is a normal graph over $S_0 \setminus B_{r_0}( p_0^+) $), along with perturbed necks $ \psi_k^{-1} (\tilde N_k)$ for $k = 0, \ldots, K-1 $.  Then glue these building blocks together using cut-off functions and matched asymptotics again as before.     This construction should be equipped with the appropriate separation and displacement parameters $\sigma_0, \ldots \sigma_{K-1}$ and $\delta_0, \ldots, \delta_{K-1}$.

\paragraph*{Step 2.} The standard Delaunay surface of mean curvature 2 and with the appropriate small neck radius can be rescaled by a factor of $r$ and translated along the geodesic $\gamma$ until there is overlap with the last perturbed sphere $\tilde {S}_K $ of the construction of Step 1.  Optimal overlap can be achieved because the neck region of this Delaunay surface is to first approximation a standard catenoid. 

To be a bit more precise with this idea, one proceeds as follows. First recall that Delaunay surfaces of mean curvature 2 are the surfaces of revolution generated by the functions $\rho_T : \R \rightarrow \R$ studied in \cite[\S 3]{mazzeopacardends}.  These functions are periodic with period $T$ and have a local minima at the integer multiples of $T$.   Introduce the parameters $d_K$ (which will be fixed once and for all below), and $\delta_K$ and $\sigma_K$ (which remain free). Define $T := 2  + \sigma_K / r$ and $\eps_K := \rho_T(0)$ and $T_K := t_K + r(d_K + \delta_K)$ where $t_K$ is the arc-length parameter for the center of the perturbed sphere $\tilde S_K$.  Now parametrize a family of translated Delaunay surfaces of mean curvature $\frac{2}{r}$ and period $2r + \sigma_K$ via 
$$ \Xi:  (t, \theta) \mapsto \left( r\rho_{T} \bigg( \frac{t - T_K }{r} \bigg) \cos(\theta) \, , \,  r \rho_{T} \bigg( \frac{t - T_K}{r} \bigg) \sin(\theta) \, , \,  t   \right) \, .$$
Note that the $t$-parameter now corresponds exactly to the arc-length parameter of $\gamma$.  Define truncations of these Delaunay surfaces via 
$$D_r^+(\sigma_K, \delta_K) := \Big\{ \Xi(t, \theta) : t \in [T_K - \omega, \infty) \: \: \mbox{and} \: \: \theta \in \Sph^1 \Big\}$$
where $\omega$ is a small number of size $\mathcal O (r \eps^{3/4})$.  

Next, one must find the optimally matching Delaunay surface.   First note that the neck size $\eps_K$ of the Delaunay surface $D_r^+(\sigma_K,  \delta_K)$ depends on $\sigma_K$.  Now as in Step 2 of the construction of the finite-length surface, the equality of the constant and logarithmic terms of the expansion of the first neck of $D_r^+(\sigma_K, 0)$ and of the asymptotic region of $\tilde {S}_K$ near $p_K^+$ determines $d_K$ and $\eps_K^+$ uniquely in terms of the remaining free parameter $\sigma_K$.  Now modify $D_r^+(\sigma_K, 0)$ by means of cut-off function as in Step 3 of of the construction of the finite-length surface so that it coincides exactly with $\tilde{S}_K$ in the region $\tilde{S}_K \cap \{ (t, x) : \frac{1}{2} r \omega  \leq \| x \| \leq r \omega \}$.  Construct also surfaces where $\delta_K \neq 0$, introducing mis-match.  Denote the family of modified surfaces by $\tilde D_r^+ (\sigma_K, \delta_K)$.  The definition of the one-ended approximate solutions is finally at hand.

\begin{defn}
	Let $K$ be given.  The one-ended surface with parameters $\sigma := \{ \sigma_0, \ldots, \sigma_{K} \}$ and $\delta := \{ \delta_0, \ldots, \delta_{K}\}$ is
$$\apsoloe := \tilde S_0' \setminus B_R(p_0^+)  \cup \left[\, \bigcup_{k=0}^K  \tilde S_k \setminus B_R(p_k^\flat) \right]  \cup \left[\, \bigcup_{k=0}^{K-1} \psi_k^{-1} (\tilde N_k) \right]  \cup \tilde D^+_r(\sigma_K, \delta_K) \, .$$
\end{defn}

\section{Function Spaces and Norms}

The constant mean curvature equation for normal perturbations of the approximate solutions  $\apsol$ constructed in the previous section will be solved for functions in weighted H\"older spaces.  The weighting will account for the fact that the geometry of $\apsol$ is nearly singular in the small neck-size limit that will be considered here.  In fact, two slightly different function spaces will be introduced.  The space $C^{k,\alpha}_\nu (\apsolf)$ will consist of all $C^{k,\alpha}_{\mathit{loc}}$ functions on $\apsolf $ where the rate of growth in the neck regions of $\apsolf$ is controlled by the parameter $\nu$.  The space $C^{k,\alpha}_{\nu, \bar \nu} (\apsoloe)$ will consist of all $C^{k,\alpha}_{\mathit{loc}}$ functions on $\apsoloe $ where the rate of growth in the neck regions of $\Sigma_r^{\mathit{OE}}(\sigma)$ is controlled by the parameter $\nu$ and the rate of growth in the asymptotic region along the axis of $\apsoloe$ is controlled by the parameter $\bar \nu$.  This latter degree of  control is necessitated by the non-compactness of $\apsoloe$.  The two types of spaces described above will collectively be denoted $C^{k,\alpha}_\ast(\apsol)$ for brevity whenever needed.

\subsection{Function Spaces and Norms for the Finite-Length Surface}

To begin, one must introduce a number of objects.  Define a weight function $\zeta_{r} : \apsol\rightarrow \R$ to achieve control of the growth of functions on $\apsol$ in the neck regions:
\begin{equation*}
	\zeta_{r}(p) := 
	\begin{cases}
		r \| x \| &\quad \psi_k(p) = (t, x) \in \tilde N_k  \:\: \mbox{with} \:\: x \in  \bar B_{R' \! /2}(0)  \:\: \mbox{for some} \:\:  k \\
		\mathit{Interpolation} &\quad \psi_k(p) = (t, x) \in \psi_k( \tilde S_k) \:\: \mbox{with} \:\:  x \in \bar B_{R'}(0) \setminus B_{R' \! /2}(0)  \:\: \mbox{for some} \:\:  k \\
		r &\quad \mbox{elsewhere} 
	\end{cases}
\end{equation*}
where the interpolation is such that $\zeta_r$ is smooth and monotone in the region of interpolation, has appropriately bounded derivatives, and is invariant under all the symmetries of $\apsol$.  For the norms themselves, first introduce the following terminology.  If $\mathcal U$ is any open subset of $\apsol$ and $T$ is any tensor field on $\mathcal U$, define 
\begin{equation*}
	| T |_{0, \mathcal U} := \sup_{x \in \,  \mathcal U} \Vert T(x) \Vert \\
	\qquad \mbox{and} \qquad [T]_{\alpha, \, \mathcal U} := \sup_{x,x' \in \, \mathcal U} \frac{\Vert T(x') - \Xi_{x,x'} (T(x)) \Vert}{\mathrm{dist}(x,x')^\alpha} \, ,
\end{equation*}
where the norms and the distance function that appear are taken with respect to the induced metric of $\apsol$, while $\Xi_{x,x'}$ is the parallel transport operator from $x$ to $x'$ with respect to this metric.  Then, for any function $f: \mathcal U \rightarrow \R$ define
$$|f|_{k, \alpha, \nu, \, \mathcal U} := \sum_{i=0}^{k}  | \zeta_r^{i-\nu} \nabla^i f |_{0, \, \mathcal U} + [ \zeta_r^{k + \alpha -\nu} \nabla^k f ]_{\alpha, \,\mathcal U} \, .$$
The norms for the finite-length and one-ended surfaces can now be defined.

Define a collection of overlapping open subsets of $\apsolf$ as follows.  Let $\bar R$ be a fixed radius such that $\bigcup_k B_{2 \bar R}(p_k^\flat)$ contains all the neck regions of $\tilde \Sigma_r^F(\sigma, \delta)$ and $\mathcal A := \apsolf \setminus \big[ \bigcup_k \bar B_{\bar R}(p_k^\flat)  \big]$ is the disjoint union of all the spherical regions of $\apsol$.  Let $\mathcal A_R := \apsolf \cap \big[ \bigcup_k B_{2 R}(p_{k}^\flat) \setminus \bar B_R (p_{k}^\flat) \big]$ for any choice of $R \in (0, \bar R)$.  
   
\begin{defn} 
	Let $\mathcal U \subseteq \apsolf$ and $\nu \in \R$ and $\alpha \in (0,1)$. The $C^{k, \alpha}_\nu$ norm of a function defined on $\mathcal U$ is given by
\begin{equation*}
	| f |_{C^{k, \alpha}_\nu (\mathcal U)} :=  |  f |_{k, \alpha, \,\mathcal U \cap \mathcal A} 
	+ \sup_{R \in (0, \bar R ) } | f |_{k, \alpha, \, \mathcal U \cap \mathcal A_R } \, . 
\end{equation*}  
\end{defn}

The function spaces that will be used in the case of the finite-length surface are simply the usual spaces $C^{k, \alpha}(\apsolf)$, but endowed with the $C^{k, \alpha}_\nu$ norm. This space will be denoted $C^{k, \alpha}_{\nu} (\apsolf)$.

\subsection{Function Spaces and Norms for the One-Ended Surface}

The definition of the weighted norm that will be used in the case of the one-ended approximate solution builds upon the norm just defined above.  Extend the collection of overlapping open subsets used above by defining the points $p_k^\flat$ for $k \geq K$ as the points of $\gamma$ upon which the neck regions of $\tilde D^+_r(\sigma_K, \delta_K)$ are centered, and then re-defining $\mathcal A$ and $\mathcal A_R$ as infinite unions over all $k \in \N$.  Furthermore, re-define the weight function $\zeta_r$ by leaving it unchanged on $\apsoloe \setminus \tilde D^+_r(\sigma_K, \delta_K)$ and defining
\begin{equation*}
	\zeta_{r}(p) := 
	\begin{cases}
		r \| x \| &\quad \psi_k(p) = (t, x) \in \tilde D^+_r(\sigma_K, \delta_K)  \:\: \mbox{with} \:\: x \in  \bar B_{R' \! /2}(0)   \\
		\mathit{Interpolation} &\quad \psi_k(p) = (t, x) \in \tilde D^+_r(\sigma_K, \delta_K) \:\: \mbox{with} \:\:  x \in \bar B_{R'}(0) \setminus B_{R' \! /2}(0)  \\
		r &\quad \mbox{elsewhere} 
	\end{cases}
\end{equation*}
Finally, for any subset $\mathcal U \subseteq \apsoloe$ and function $f \in C^{k, \alpha}_{\mathit{loc}}(\mathcal U)$ define the norm $|f|_{k, \alpha, \nu, \, \mathcal U}$ as above.  A second collection of overlapping open subsets of $\tilde D^+_r(\sigma_K, \delta_K)$ will be now be introduced.  Let $\bar T := t_{K/2}$ be the arc-length parameter corresponding to the point $p_{N/2}$ and for any choice of $T \geq \bar T$ define $\mathcal C_T := \apsoloe \cap \{ (x, t) \in \R^2 \times \R : t \in [T, T+ r] \}$.  Also define $\mathcal C := \apsoloe \setminus C_{\bar T}$. 

\begin{defn} 
	\label{def:normoe}
	Let $\mathcal U \subseteq \apsoloe$ and $\nu, \bar \nu \in \R$ and $\alpha \in (0,1)$. The $C^{k, \alpha}_{\nu, \bar \nu}$ norm of a function defined on $\mathcal U$ is given by
\begin{align*}
	| f |_{C^{k, \alpha}_{\nu, \bar \nu} (\mathcal U)} &:= |  f |_{k, \alpha,\nu, \,\mathcal U \cap \mathcal A \cap \mathcal C} +   \sup_{T > \bar T}  \, \, T^{-\bar \nu} \left( |  f |_{k, \alpha, \nu, \,\mathcal U \cap \mathcal A \cap \mathcal C_T} 
	+ \sup_{R \in (0, R_0 ) }  | f |_{k, \alpha,\nu , \, \mathcal U \cap \mathcal A_R \cap \mathcal C_T } \right) \, . 
\end{align*}  
\end{defn}

The function spaces that will be used in the case of the one-ended surface are the spaces $C^{k, \alpha}_{\nu, \bar \nu} (\apsoloe) := \{ f \in C^{k, \alpha}_{\mathit{loc}} (\apsoloe) : |f|_{C^{k,\alpha}_{\nu, \bar \nu}} < \infty \}$ endowed with the $C^{k, \alpha}_{\nu, \bar \nu}$ norm.

\begin{rmk}
	At first glance, the norm of Definition \ref{def:normoe} looks different from the norm used to study the linearized mean curvature operator of near-degenerate Delaunay surfaces in \cite[\S 4]{mazzeopacardends}.  This is because the norm in \cite[\S 4]{mazzeopacardends} is defined using a different parametrization (the $s$-parameter).  However, it is straightforward to check, using the estimates of \cite[\S 4]{mazzeopacardends} relating the $s$-parameter to the arc-length parameter, that the norm in \cite[\S 4]{mazzeopacardends} is equivalent to the $C^{k,\alpha}_{0, \bar \nu}$ norm defined above.
\end{rmk}

\section{Estimates of the Mean Curvature of the Approximate Solutions}
\label{sec:meancurvest}

This section estimates the amount which the approximate solutions $\apsol$ deviate from being CMC surfaces.  This is accomplished by estimating $H \big[\apsol \big] - \frac{2}{r}$ in the $C^{0,\alpha}_\ast$ norm when $\nu \in (1,2)$ for the finite-length surface and $(\nu, \bar \nu) \in (1,2) \times (-1, 0)$ for the one-ended surface.  In the estimates that follow, set $ \eps := \max \{ \eps_k^\pm, \eps_k \}$ and $\delta := \max \{ \delta_k \}$ and recall that $r_k = \mathcal O(r \eps^{3/4})$.

\subsection{The Estimate for the Finite-Length Surface}

The result for the finite-length surface $\apsolf$ is as follows.

\begin{prop}
	\label{prop:meancurvest}
	Suppose $\nu \in (1,2)$.  The mean curvature of $\apsolf$ satisfies the estimate
	$$\left|  H - \frac{2}{r} \right|_{\coan} \leq C \max \big\{ r^{3 - \nu}  , \, r^{5-\nu} \eps^{1/2 - 3\nu/4} , \, r^{1-\nu} \eps^{3/2 - 3\nu/4} , \delta r^{1-\nu} \eps^{1 - 3\nu/4} \big\} $$
	for some constant $C$ independent of $r$, $\eps$, $\delta$ and $K$.
\end{prop}

\begin{proof} There are several steps: the first two steps are to derive pointwise estimates for the mean curvature and second fundamental form of $\apsolf$ with respect to the Euclidean background metric in the spherical and neck regions, respectively; the third step is to convert these into pointwise estimates for the mean curvature with respect to the actual background metric; and the fourth step is to compute the desired $\coan$ norm of $H \big[\apsolf \big] - \frac{2}{r}$.  Only the weighted $C^0$ norm will be estimated explicitly below since the calculations for the weighted H\"older coefficient are very similar.  Finally, all estimates computed below are independent of $K$.

\paragraph*{Step 1.} The first step it to find a pointwise estimate for $\mathring B := \mathring B \big[\apsolf \big]$ and $\mathring H := \mathring H \big[\apsolf \big]$, the Euclidean second fundamental form and mean curvature, respectively, in the spherical region $\tilde S_k$ of $\apsolf$.  The key to this estimate is to use the formul\ae\ \eqref{eqn:quadsecondff} and \eqref{eqn:quadmeancurv} for $\mathring B$ and $\mathring H$ in terms of the graphing function $rG$ and the second fundamental form and mean curvature of $S_k$ in conjunction with the estimates from Section \ref{sec:euclideancalcs}.  The result is
\begin{equation}
\begin{aligned}
	\| \mathring B \| &\leq \frac{C}{r} \left( 1 +  | G | + r\| \mathring  \nabla G\| + r^2\| \mathring \nabla^2 G \big\| \right)  \\
	\Big| \mathring H - \frac{2}{r} \Big| &\leq \frac{C}{r} \left(  | G | + r\| \mathring  \nabla G\| + r^2\| \mathring \nabla^2 G\| \right)^2 \, .
\end{aligned}
\end{equation}
The reason the estimate for $\big| H - \frac{2}{r} \big|$ is so much better is because $\mathring {\mathcal L} (G ) = 0$ in the region being considered.  
 
To proceed with the estimate for $\big| H - \frac{2}{r} \big|$, it is thus necessary to estimate $| G | + r\| \mathring  \nabla G\| + r^2\| \mathring \nabla^2 G\|$.  First, in the part of $S_k$ near the points $p_\ast^\flat$ (here $\ast$ is $k$ or $ k-1$ as appropriate), where the distance to these points can be as small as $\mathcal O (r \eps^{3/4})$, one can use the Taylor series expansion of $G$ to derive the estimate
$$| G - \eps \, c \log( \rho) | + \rho \| \mathring \nabla G \| + \rho^2 \| \mathring \nabla^2 G \| \leq C \eps   \, ,$$
where $\rho :=  \mathrm{dist}( p_\ast^\flat \, , \cdot)$ is the distance function to $p_\ast^\flat$ with respect to the standard, unit-radius, induced metric of the sphere.  Next, in the part of $\tilde S_k$ that is an $\eps$-independent distance away from these points, $G$ satisfies the estimate
$$|G| + \| \mathring \nabla G \| + \| \mathring \nabla^2 G \|  \leq C \eps \, .$$  
In both of these estimates,  $c$ and $C$ are numerical constants.  Therefore near the points $p_\ast^\flat$ one has 
$$\| \mathring B \| \leq \frac{C r \eps}{\rho^2} \qquad \mbox{and} \qquad \Big| \mathring H - \frac{2}{r} \Big| \leq \frac{C r^3 \eps^2}{\rho^4} $$
while in the remainder of $\tilde S_k$ one has
$$\| \mathring B \| \leq \frac{C}{r}  \qquad \mbox{and} \qquad \Big| \mathring H - \frac{2}{r} \Big| \leq \frac{C \eps^2}{r}  \, .$$

\paragraph*{Step 2.} The next step is to look inside one of the scaled normal coordinate neighbourhoods used in the definition of the necks.  Again, the calculations  are performed with respect to the scaled Euclidean metric, and computed in the transition region $T_k := \tilde N_k \cap  \{ (t, x) : x \in B_{\eps_k^{3/4}}(0) \setminus B_{\eps_k^{3/4}/2}(0) \}$ as well as the neck region $N_k := \tilde N_k \cap  \{ (t, x) : x \in B_ {\eps_k^{3/4} /2}(0) \}$ itself.   Of course, the estimates in $N_k$ are extremely straightforward since $N_k$ is exactly the standard catenoid for which  $\mathring H = 0$ and $\| \mathring B \| = \sqrt{2} \, \eps \, \| x \|^{-2 }$.  Hence the challenge lies in estimating $\mathring H$ and $\mathring B$ in $T_k$.

The transition region $T_k$ is the graph of the function $\tilde F := \eta F + (1-\eta) G$ as in \eqref{eqn:interpfn}, where $\eta$ is a cut-off whose derivative is supported in $\frac{1}{2} \eps_k^{3/4} \leq \| x \|  \leq\eps_k^{3/4}$, while $F(x) := \eps_k \arccosh( \|x \| / \eps_k) + d_k + \delta_k$ and $G$ has an asymptotic expansion that matches the asymptotic expansion of $F$ except for the mismatch introduced by $\delta_k$.  Now,  the scaled Euclidean second fundamental form and mean curvature of a graph are 
\begin{align*}
	\mathring D \mathring B_{ij} &=  \mathring \nabla^2_{ij} \tilde F  \\[1ex]
	\mathring D \mathring H &=  \left( \delta^{ij} - \frac{\mathring \nabla^i \tilde F \,\, \mathring \nabla^j \tilde F}{\mathring D^2 } \right) \mathring \nabla^2_{ij} \tilde F
\end{align*}
where $\mathring D := \big( 1 + \| \mathring \nabla \tilde F \|^2 \big)^{1/2}$.  Write $\tilde F := u + F$ where $u := (1-\eta)(G-F) $ and observe that
$$ u= \eps_k \delta_k (1-\eta) + (1-\eta) (\hat G - \hat F)$$
for constants $c$ and $C $ of size $ \mathcal O(\eps)$ and functions $\hat G$ and $\hat F$ of size $\mathcal O( \| x \|^2)  + \mathcal O(\eps^3 \| x \|^{-2} )$.  When  $\| x \| = \mathcal O(\eps^{3/4})$ one has 
\begin{gather*}
	u - \eps_k \delta_k (1 - \eta) = \mathcal O(\eps^{3/2}) \\
	\mathring \nabla u + \eps_k \delta_k \mathring \nabla \eta = \mathcal O(\eps^{3/4}) \\
	\mathring \nabla^2 u + \eps_k \delta_k \mathring \nabla^2 \eta = \mathcal O(1) 
\end{gather*}
whereas $\| \mathring \nabla F \| = \mathcal O(\eps^{1/4})$ and $\| \mathring \nabla^2 F \| = \mathcal O(\eps^{-1/2})$. Thus plugging $\tilde F$ into the expression for $\mathring B$ and estimating yields 
$$\| \mathring B \| \leq C \eps^{-1/2} \, .$$
Next, plugging $\tilde F$ into the expression for $\mathring H$ yields
\begin{align*}
	D \mathring H  &= \mathring \Delta u + \left( \frac{1}{\sqrt{1 + \| \mathring \nabla F \|^2} } - \frac{1}{\sqrt{1 + \| \mathring \nabla F  + \mathring \nabla u \|^2} } \right) F_{,i} F_{,j} F_{,ij} \\[1ex]
	&\qquad - \frac{u_{,i} u_{,j} u_{,ij} + u_{,i} u_{,j} F_{,ij} + 2 u_{,i} F_{,j} u_{,ij} + F_{,i} F_{,j} u_{,ij} + 2 F_{,i} u_{,j} F_{,ij} }{\sqrt{1 + \| \mathring \nabla F  + \mathring \nabla u \|^2} }
\end{align*}
where $\mathring \Delta F - (1 + \| \mathring \nabla F \|^2)^{-1} F_{,i} F_{,j} F_{,ij} \equiv 0$ has been used, which holds since $F$ is the graphing function for the catenoid which has zero mean curvature.  Therefore one can estimate
$$| H - 2 + \eps_k \delta_k \mathring \Delta \eta | \leq C$$
where $C$ is independent of $\eps$.

\paragraph*{Step 3.} The next step is to compute the pointwise norm of $H := H[\apsolf]$ with respect to the actual background metric. Lemma \ref{prop:approxexpansions} relates $H$ to $\mathring H$ and $\mathring B$, and the form of the relationship that is germane to the current derivation is expressed most succinctly in equation \eqref{eqn:meancurvf}.  Thus to proceed, one must substitute the estimates from Steps 1 and 2 into this formula and estimate the curvature terms as in Section \ref{sec:perturbedcalcs}.  Again, this should be done in the different regions identified above.  Consider first  the spherical region $\tilde S_k$, where $\| Y \| = \mathcal O(r)$ and
\begin{align*}
	\left| H - \frac{2}{r} \right| &\leq    \left| \mathring H - \frac{2}{r} \right| + |R(0)| \| \mathring B \| + | \bar R(0) \| \\
	&\leq  \left| \mathring H - \frac{2}{r} \right| + C \big( \|Y \| + \|Y\|^2 \| \mathring B \|\big) \\[0.5ex]
	&\leq 
	\begin{cases}
		C \left( \dfrac{r^3 \eps^2}{\rho^4} + r \right) \\[3ex]
		C \left( \dfrac{\eps^2}{r} + r \right) \, .
	\end{cases} 
\end{align*}
near $p_\ast^\flat$ and away from $p_\ast^\flat$, respectively.

Consider now the neck and transition regions $\tilde N_k$.  First, by reversing the scaling used in Step 2, the second fundamental form and mean curvature in $\tilde N_k$, measured with respect to the Euclidean metric, satisfy the estimates
$$\left| \mathring H - \frac{2}{r} + \eps_k  \delta_k \mathring \Delta \eta \right| \leq \frac{C}{r}  \qquad \mbox{and} \qquad \| \mathring B \| \leq \frac{C}{r \eps^{1/2}} \, .$$
Now these estimates can be plugged into the expansion of the mean curvature as above, except that $Y$ is now the position vector field of $\psi^{-1}(\tilde N_k)$ relative to the center of the normal coordinate chart used in the construction of the neck --- namely the point $p_k^\flat$.  Hence $\|Y \|$ is uniformly bounded by $\max \{ r|G|, r|F| \} \leq \mathcal O (r \eps | \log(\eps)|)$.  With this in mind, one obtains the estimate
\begin{align*}
	\left| H - \frac{2}{r} \right| 
	&\leq 
	\begin{cases}
		\dfrac{C}{r}\! \left( 1 + r^2 \eps |\log(\eps) |  + r^2 \eps^{3/2} | \log(\eps) |^2  \right) \\[3ex]
		\dfrac{C}{r}\!  \left( 1 + r^2 \eps |\log(\eps) |  + \dfrac{r^2 \eps^3 | \log(\eps) |^2}{ \| x \|^2}   + r \eps_k \delta_k | \mathring \Delta \eta | \right) \, .
	\end{cases} 
\end{align*}
in the transition region of the neck and in the neck itself, respectively.
\begin{rmk}
	In all of these estimates, the $\mathcal H$ term from Lemma \ref{prop:approxexpansions} is always negligible.
\end{rmk}

\paragraph*{Step 4.} The remaining task is to estimate $| H - \frac{2}{r}|$ in the $C^{0,\alpha}_{\nu - 2}$ norm for $\nu \in (1,2)$.  This estimate will be derived by estimating the supremum of $\zeta_r^{2-\nu} | H - \frac{2}{r}|$ in the three regions of $\apsolf$ identified in the previous steps.  For now, the only assumption that will be made about $\eps$ is that $\eps \ll r$. First, consider a spherical region $\tilde S_k$ away from the points $p_\ast^\flat$.  In this region $\zeta_r \equiv r$ so that
$$\zeta_r^{2-\nu} \left| H - \frac{2}{r} \right| \leq C r^{3-\nu}  \, .$$
Also, in a spherical region $\tilde S_k$ near the points $p_\ast^\flat$, one has instead $\zeta_r \approx \mathrm{dist} ( p_\ast^\flat, \cdot) = r \rho$.   Thus
\begin{align*}
	 \zeta_r^{2-\nu} \left| H - \frac{2}{r} \right| &\leq \sup_{\rho \in [\eps^{3/4}, R']} C (r \rho)^{2-\nu} \left( r + \frac{r^3 \eps^2}{\rho^4} \right) \\
	 &\leq C r^{3-\nu} \big( \eps^{3(2 - \nu)/4} + r^{2} \eps^{1/2 - 3 \nu/4} \big) 
\end{align*}
Next, in the neck region $\tilde N_k$ one has $\zeta_r(x) = r \| x \|$ in the local coordinates used to define $\tilde N_k$.  Hence
\begin{align*}
	\zeta_r^{2-\nu} \left| H - \frac{2}{r} \right| &\leq \sup_{\eps \leq \|x \| \leq \eps^{3/4}} C r^{1-\nu} \|x \|^{2-\nu} \left( 1 + \frac{\eps^2}{\| x \|^2}  + r^2 \eps |\log(\eps) |  + \frac{r^2 \eps^3 | \log(\eps) |^2}{ \| x \|^2}  + \eps_k \delta_k|  \mathring \Delta \eta (x)| \right) \\
	&\leq C r^{1-\nu} \big( (1 + r^2 \eps | \log(\eps) | ) \eps^{3(2 - \nu)/4} + \eps^{2-\nu} + r^2 \eps^{3-\nu} | \log(\eps)|^2 + \delta_k \eps^{-1/2} \big) \\
	&\leq C r^{1-\nu} \left( \eps^{3/2 - 3\nu/4} + \delta_k \eps^{1 - 3\nu/4} \right) \, .
\end{align*}
Therefore consolidating these three estimates yields
$$\left|  H - \frac{2}{r} \right|_{C^{0}_{\nu - 2}} \leq C \max \big\{ r^{3 - \nu}  , \, r^{5-\nu} \eps^{1/2 - 3\nu/4} , \, r^{1-\nu} \eps^{3/2 - 3\nu/4} , \delta_k r^{1-\nu} \eps^{1 - 3\nu/4}  \big\}$$
which is the desired weighted supremum norm estimate.  The full $\coan$ estimate follows once the estimate for the H\"older coefficient has been computed.  As indicated above, this computation is more involved but very similar, and yields the same result.
\end{proof}

\subsection{The Estimate for the One-Ended Surface}

The estimate of the mean curvature of the one-ended surface $\apsoloe$ is a very straightforward extension of the results obtained in the previous section.  In fact, the result is the same.

\begin{prop}
	\label{prop:meancurvestoe}
	Suppose $\nu \in (1,2)$ and $\bar \nu \in (-1, 0)$ is sufficiently close to zero.  The mean curvature of $\apsoloe$ satisfies the estimate
	$$\left|  H - \frac{2}{r} \right|_{\coann} \leq C \max \big\{ r^{3 - \nu}  , \, r^{5-\nu} \eps^{1/2 - 3\nu/4} , \, r^{1-\nu} \eps^{3/2 - 3\nu/4}, \delta r^{1-\nu} \eps^{1 - 3\nu/4}   \big\} $$
	for some constant $C$ independent of $r$, $\eps$, $\delta$ and $K$.
\end{prop}

\begin{proof}
	The computations of Proposition \ref{prop:meancurvest} give the estimate of the mean curvature of the part $\apsoloe$ constructed from spheres and catenoids.  It thus remains only to compute the estimate of the mean curvature of $\tilde D^+_r(\sigma_K, \delta_K)$ (actually, the part of $\tilde D^+_r(\sigma_K, \delta_K) $ that is an un-perturbed Delaunay surface).  Once again, the key is to use Lemma \ref{prop:expansions} in the form of equation \eqref{eqn:meancurvf}, but this time realizing that the mean curvature of $\tilde D^+_r(\sigma_K, \delta_K) $ with respect to the Euclidean background metric in a tubular neighbourhood of $\gamma$ is exactly equal to $\frac{2}{r}$.   Let $p$ be any point on $\gamma$ and let $p'$ be any point on $\tilde D^+_r(\sigma_K, \delta_K) \cap B_r(p)$.  Then by equation \eqref{eqn:meancurvf}, 
$$ \left| H(p') - \frac{2}{r} \right| \leq C |R(p)|  \big( \| Y(p, p') \|^2 \| \mathring B(p') \| + \| Y(p, p') \| \big)
$$
where $R$ is an expression that is linear in the components of the ambient Riemannian curvature at $p$ and $Y(p, p')$ is the position vector field of $p'$ with respect to $p$.   The constant $C$ is independent of $r$ and $\eps$.  Since the curvature is exponentially decaying along $\gamma$, the same estimates from Proposition \ref{prop:meancurvest} continue to hold and yield the estimate needed here.  
\end{proof}

\section{The Solution up to Finite-Dimensional Error}
\label{sec:projsolution}

\subsection{The Finite-Length Surface}

\subsubsection{Strategy} 
\label{sec:strategy}

Let $\mu : \ctan (\apsolf) \rightarrow \mathit{Emb}( \apsolf, M)$ be the exponential map of $\apsolf$ in the direction of the unit normal vector field of $\apsolf$ with respect to the backgroung metric $g$.  Hence $ \mu_{rf} \big(\apsolf \big)$  is the scaled normal deformation of $\apsol$ generated by $f \in \ctan (\apsolf)$.  The equation
\begin{equation}
	\label{eqn:tosolve}
	H \big[ \mu_{rf} \big(\apsolf \big) \big]  = \frac{2}{r}
\end{equation}
selects $f \in \ctan (\apsolf)$ so that $\mu_{rf}( \apsolf)$ has constant mean curvature equal to $\frac{2}{r}$.   In addition, the function $f$ will be assumed symmetrical with respect to all the symmetries satisfied by $\apsolf$.  Using a fixed-point argument together with a suitable choice of weight parameters,  the equation \eqref{eqn:tosolve} will be solved up to a \emph{finite-dimensional error term}.  This means that a solution of
$$H \big[ \mu_{rf} \big(\apsolf \big) \big]  = \frac{2}{r} + \mathcal E$$
will be found, where $\mathcal E$ belongs to a finite-dimensional subspace of functions that will be denoted $\tilde{\mathcal W}^F$ and specified below.  At first glance, the error $\mathcal E$ will come from terms in the solution procedure that are not sufficiently small.  However, the true reason for the presence of $\mathcal E$ is geometric and will be explained in 
Section \ref{sec:balancing}, where we show how to eliminate it. 

To begin, write $H [  \mu_{rf} \big(\apsolf \big)] = H [  \apsol] + \mathcal L (rf) + \mathcal Q(rf) + \mathcal H(rf)$ where $\mathcal L$ is the linearized mean curvature operator, $\mathcal Q$ is the quadratic remainder part of the mean curvature and $\mathcal H$ is the small error term as in equation \eqref{eqn:meancurvfnexp}.  The first step is to construct a suitably bounded parametrix  $\mathcal R :\coan(\apsolf) \rightarrow \ctan (\apsolf)$ satisfying $ \mathcal L \circ \mathcal R = \mathit{Id} + \mathcal E$ where $\mathcal E$ maps into $\tilde{\mathcal W}^F$.  Now the \emph{Ansatz} $f :=  \frac{1}{r} \mathcal R \big( w - H \big[\apsolf \big] + \frac{2}{r}  \big)$ transforms the equation \eqref{eqn:tosolve} into the fixed-point problem
\begin{align}
	\label{eqn:cmapf}
	w & = -  {\mathcal Q} \circ  \mathcal R \left(w -  H \big[\apsolf \big]  + \frac{2}{r} \right)  -  {\mathcal H} \circ  \mathcal R \left(w -  H \big[\apsolf \big]  + \frac{2}{r} \right) \, .
\end{align}
up to the finite-dimensional error term in $\tilde{\mathcal W}^F$.  The remaining task is to show that the mapping $\mathcal N_r : \coan(\apsolf) \rightarrow \coan(\apsolf)$ given by the right hand side of \eqref{eqn:cmapf} is a contraction mapping onto a neighbourhood of zero containing $H \big[\apsolf \big]  - \frac{2}{r}$.  Once this is done, then one has solved the equation \eqref{eqn:tosolveoe} up to a term in $\tilde{\mathcal W}^F$. 

\subsubsection{The Linear Analysis}
\label{sec:linanalysis}
We now find a parametrix $\mathcal R$ satisfying $\mathcal L \circ \mathcal R = \mathit{id} + \mathcal E$ where $\mathcal E$ has finite rank. In each normal coordinate chart, $\mathcal L$ is close to the linearized mean curvature operator with respect to the Euclidean metric $\mathring {\mathcal L} (f) := \mathring \Delta f + \| \mathring B \|^2 f$. Patching together inverses for the latter operator gives a parametrix with an error which decomposes into terms which are genuinely small, and those which together constitute $\mathcal E$.  

First, denote $\mathit{Ann}_\tau (p) := B_\tau(p) \setminus B_{\tau/2}(p)$ and define subsets of $\apsolf$ by
\begin{align*}
	\mathcal S_k^{\, \tau} &:= \apsolf \setminus \big[ B_\tau (p_k^\flat) \cup B_\tau (p_{k-1}^\flat) \big] \\
	\mathcal T_{k, +}^{\, \tau} &:= \apsolf \cap \mathit{Ann}_\tau(p_k^\flat) \\
	\mathcal T_{k, -}^{\, \tau} &:= \apsolf \cap \mathit{Ann}_\tau(p_{k-1}^\flat) \\
	\mathcal N_k^{\, \tau} &:= \apsolf \cap B_\tau (p_k^\flat) \, .
\end{align*}
Now define the smooth, monotone cut-off functions  
\begin{align*}
	\chi^{\tau}_{\mathit{neck}, k} (x) &:= 
	\begin{cases}
		1 &\qquad x \in \mathcal N^{\, \tau}_k \\
		\mbox{Interpolation} &\qquad x \in  	\mathcal T_{k, +}^{\, \tau} \cup \mathcal T_{k+1, -}^{\, \tau}  \\
		0 &\qquad \mbox{elsewhere}
	\end{cases} \\
	\chi^{\tau}_{\mathit{ext}, k} (x) &:= 
	\begin{cases}
		1 &\qquad x \in \mathcal S^{\, \tau}_k \\
		\mbox{Interpolation} &\qquad x \in \mathcal T_{k, +}^{\, \tau} \cup \mathcal T_{k, -}^{\, \tau}   \\
		0 &\qquad \mbox{elsewhere}	
	\end{cases}	
\end{align*}
so that $\sum_{k} \chi^{\tau}_{\mathit{ext}, k} + \sum_{k} \chi^{\tau}_{\mathit{neck}, k} = 1$ for all $\tau$ and all cut-off functions are invariant with respect to all symmetries satisfied by $\apsolf$.   

\begin{prop}
	\label{prop:linest}
	Let $\nu \in (1,2)$.  There is an operator $\mathcal R: \coan(\apsolf) \rightarrow \ctan(\apsolf)$  that satisfies $\mathcal L \circ \mathcal R = \mathit{id} - \mathcal E$ where $\mathcal E : \coan(\apsolf) \rightarrow  \tilde{\mathcal W}^F$.  Here $\tilde{\mathcal W}^F$ is a finite-dimensional space that will be defined below.  The estimates satisfied by $\mathcal R$ and $\mathcal E$ are
	$$| \mathcal R(w) |_{\ctan} + |\mathcal E(w)|_{C^{2,\alpha}_0} \leq C | w |_{\coan}$$
	for all $w \in \coan(\apsolf)$, where $C$ is a constant independent of $r$, $\eps$ and $\delta$. 
\end{prop}

\begin{proof}
Let $w \in \coan(\apsolf)$ be given.  The task at hand is to solve the equation $\mathcal L(u) = w + \mathcal E(w)$ for a function $u \in \ctan(\apsolf)$ and an error term $\mathcal E(w) \in \tilde{\mathcal W}^F$.  To begin, introduce four radii $\tau_{1} < \tau_{2} < \tau_{3} < \tau_{4}  \ll r$ with the property that the supports of the gradients of the cut-off functions $\chi_{\ast}^{\tau_i}$ and $\chi_{\ast}^{\tau_j}$ do not overlap for $i \neq j$.   These radii will need to be further specified; and this will be done in the course of the proof below.

\paragraph*{Step 1.} Let $w_{\mathit{neck}, k} := w \chi^{\tau_{3}}_{\mathit{neck}, k}$.  This function  has compact support in $\mathcal N_k^{\, \tau_{3}}$ and thus can be viewed as a function of compact support on the standard scaled catenoid.  Now consider the equation $\mathcal L_N (u) =  w_{\mathit{neck}, k}$ on the standard scaled catenoid, where $\mathcal L_N$ is the linearized mean curvature operator of the scaled catenoid $r\eps_k N$ with respect to the Euclidean metric.  Then the pull-back of $\mathcal L$ to the standard scaled catenoid is a small perturbation of $\mathcal L_N$.  By the theory of the Laplace operator on asymptotically flat manifolds, the operator $\mathcal L_N$ is surjective onto $\coan(r \eps_k N)$ when $\nu \in (1,2)$.  Hence there is a solution $u_{\mathit{neck}, k}$ as desired, satisfying the estimate $|u_{\mathit{neck}, k}|_{\ctan} \leq C |w|_{\coan}$ where these norms can be taken as the pull-backs of the weighted H\"older norms being used to measure functions on $\apsolf$.  Finally, extend $u_{\mathit{neck}, k}$ to all of $\apsolf$ by the definition $\bar u_{\mathit{neck}, k} := u_{\mathit{neck}, k} \chi^{\tau_{4}}_{\mathit{neck}, k}$ and set $\bar u_{\mathit{neck}} := \sum_k \bar u_{\mathit{neck}, k}$.

\paragraph*{Step 2.} Define $w_{\mathit{ext}, k} := \big( w - \mathcal L( \bar u_{\mathit{neck}}  )\big) \chi_{\mathit{ext}, k}^{\tau_{2}}$.  Then $w_{\mathit{ext}, k}$ is a function of compact support on $\mathcal S_k^{\tau_{2}}$ and thus can be viewed as a function of compact support on the sphere $S_k \setminus \{ p_k^+, p_k^-\}$.  Now consider the equation $\mathcal L_S (u) = w_{\mathit{ext}, k}$, where $\mathcal L_S$ is the linearized mean curvature operator of the sphere of radius $r$ with respect to the Euclidean metric.  Then the pull-back of $\mathcal L$ to the sphere is a small perturbation of $\mathcal L_S$.  It is not \emph{a priori} possible to solve the equation $\mathcal L_S (u_{\mathit{ext}, k}) = w_{\mathit{ext}, k}$ on $S_k$ because of the one-dimensional kernel of $\mathcal L_S$ that is invariant with respect to the symmetries that have been imposed.  A basis for the kernel is given by the function $J_k : S_k \rightarrow \R$ defined by $J_k := \mathring g( \mathring N, \frac{\partial}{\partial t})$ where $\mathring g$ is the Euclidean metric and $\mathring N$ is the unit normal vector of $S_k$.  Let $w_{\mathit{ext}, k}^\perp := w_{\mathit{ext}, k} - \langle w_{\mathit{ext}, k}, J_k \rangle J_k$ be the projection of $w_{\mathit{ext}, k}$ to the orthogonal complement of $\mathrm{span} \{ J_k \}$ with respect to the Euclidean $L^2$-inner product denoted by $\langle \cdot, \cdot \rangle$.  Now there is a solution of the equation $\mathcal L_S(u_{\mathit{ext}, k}) = w_{\mathit{ext}, k}^\perp$ satisfying the estimate $|u_{\mathit{ext}, k} |^\ast_{C^{2, \alpha}} \leq C | w_{\mathit{ext}, k}^\perp|^\ast_{C^{0, \alpha}}$ where $| \cdot |_{C^{k,\alpha}}^\ast$ is the scale-invariant $C^{k, \alpha}$ norm where derivatives and the H\"older coefficient are weighted by appropriate factors of $r$.   Furthermore, one has $| w_{\mathit{ext}, k}^\perp|^\ast_{C^{0, \alpha}} \leq C | w_{\mathit{ext}, k}|^\ast_{C^{0, \alpha}} \leq C \tau_{2}^{\nu - 2} |w|^\ast_{\coan}$.

A modified solution satisfying weighted estimates can be obtained as follows.  First, write $u_{\mathit{ext}, k} := v_{\mathit{ext}, k} + a_k^+ \eta_k^+ + a_k^- \eta_k^-$ where $v_{\mathit{ext}, k}$ satisfies $|v_{\mathit{ext}, k}(x)| \leq C \tau_{2}^{\nu - 2} \mathrm{dist}(x, p_k^\ast)^2 | w|_{\coan}$ near $p_k^\ast$, while the function $\eta_k^\pm$ equals one near $p_k^\pm$ and vanishes a small but $\eps$-independent distance away from these points, and $a^\pm \in \R$ satisfies $|a^\pm| \leq C \tau_{2}^{\nu - 2} |w|_{\coan}$.  This decomposition is achieved by studying the Taylor series expansion of $u_{\mathit{ext}, k}$ near $p_k^\pm$ and using the symmetries satisfied by $u_{\mathit{ext}, k}$.  Finally, by adding the correct multiple of $ J_k$ to $u_{\mathit{ext}, k}$ on each $S_k$, one can arrange to have $a_k^- = 0$ for every $k \geq 1$ and $a_0^\pm = 0$ by symmetry.  To extend the functions $u_{\mathit{ext}, k}$ to all of $\apsolf$, define $\bar u_{\mathit{ext}} := \bar v_{\mathit{ext}} + A$ where $\bar v_{\mathit{ext}} := \sum_{k} \chi^{\tau_{1}}_{\mathit{ext}, k} v_{\mathit{ext},k} $ and $  A := \sum_k a_k^+ \eta_k^+ \chi_{\mathit{neck}, k}^{\tau_{1}}$.  Then one has the estimate $|\bar v_{\mathit{ext}}|_{\ctan} +  \tau_{2}^{2 - \nu} |A|_{\ctan} \leq C | w |_{\coan}$. 

\paragraph*{Step 3.} Let $u^{(1)} := \bar v_{\mathit{ext}} + \bar u_{\mathit{neck}}$ and $\mathcal E^{(1)}(w) := - \sum_k \chi_{\mathit{ext}, k}^{\tau_{1}} \langle w_{\mathit{ext}, k}, J_k \rangle J_k - \sum_k \chi_{\mathit{ext}, k}^{\tau_{1}} \mathcal L_S(A)$.  By collecting the estimates from Steps 1 and 2, one has $|u^{(1)}|_{\ctan} \leq C |w|_{\coan}$.   The claim is that
\begin{equation}
	\label{eqn:iteration}
	|\mathcal L (u^{(1)}) - w - \mathcal E^{(1)}(w) |_{\coan} \leq \theta |w|_{\coan} \qquad \mbox{and} \qquad | \mathcal E^{(1)}(w)|_{C^{0,\alpha}_2} \leq \; C \tau_{2}^{\nu - 2} |w|_{\coan}
\end{equation}
where $\theta$ can be made as small as desired by adjusting $\tau_{1}, \ldots, \tau_{4}$ and $\eps$ suitably.  The consequence is that one can iterate Steps 1 and 2 to construct sequences $u^{(n)}$ and $\mathcal E^{(n)}(w)$ that converge to $u := \mathcal R(w)$ and $\mathcal E(w)$ respectively, satisfying the desired bounds.

Therefore to complete the proof of the proposition, it remains to compute the estimates given in \ref{eqn:iteration}.  The idea is to exploit the fact that $\mathcal L$ differs very little from $\mathcal L_N$ and $\mathcal L_S$ in the regions where $u^{(1)}$ equals $\bar v_{\mathit{ext}}$ and $\bar u_{\mathit{neck}}$, all while taking into account the effects of the cut-off functions.  With this in mind, and using the notation $[ \mathcal L, \chi](u) := \mathcal L (\chi u) - \chi \mathcal L(u)$, one has
\begin{align*}
	\mathcal L (u^{(1)}) &= (\mathcal L - \mathcal L_S) ( \bar v_{\mathit{ext}}) + \sum_k [ \mathcal L_S, \chi_{\mathit{ext}, k}^{\tau_{1}}] (v_{\mathit{ext}, k}) + \sum_k \chi_{\mathit{ext}, k}^{\tau_{1}} \mathcal L_S( v_{\mathit{ext}, k}) + \mathcal L (\bar u_{\mathit{neck}}) \\
	&=  (\mathcal L - \mathcal L_S) ( \bar v_{\mathit{ext}}) + \sum_k [ \mathcal L_S, \chi_{\mathit{ext}, k}^{\tau_{1}}] (v_{\mathit{ext}, k}) + \sum_k \chi_{\mathit{ext}, k}^{\tau_{1}} w_{\mathit{ext}, k}  \\
	&\qquad - \sum_k \chi_{\mathit{ext}, k}^{\tau_{1}} \langle w_{\mathit{ext}, k}, J_k \rangle J_k - \sum_k \chi_{\mathit{ext}, k}^{\tau_{1}} \mathcal L_S(A) + \mathcal L(\bar u_{\mathit{neck}}) \\
	&= (\mathcal L - \mathcal L_S) ( \bar v_{\mathit{ext}}) + \sum_k [ \mathcal L_S, \chi_{\mathit{ext}, k}^{\tau_{1}}] (v_{\mathit{ext}, k}) + \sum_k  \chi_{\mathit{ext}, k}^{\tau_{2}} ( w - \mathcal L( \bar u_{\mathit{neck}}) )  \\
	&\qquad - \sum_k \chi_{\mathit{ext}, k}^{\tau_{1}} \langle w_{\mathit{ext}, k}, J_k \rangle J_k - \sum_k \chi_{\mathit{ext}, k}^{\tau_{1}} \mathcal L_S(A) + \mathcal L(\bar u_{\mathit{neck}}) \\
	&= (\mathcal L - \mathcal L_S) ( \bar v_{\mathit{ext}}) + \sum_k [ \mathcal L_S, \chi_{\mathit{ext}, k}^{\tau_{1}}] (v_{\mathit{ext}, k}) + \sum_k \chi_{\mathit{ext}, k}^{\tau_{2}} w +  \sum_k \chi_{\mathit{neck}, k}^{\tau_{2}}\mathcal L( \bar u_{\mathit{neck}})   \\
	&\qquad - \sum_k \chi_{\mathit{ext}, k}^{\tau_{1}} \langle w_{\mathit{ext}, k} , J_k \rangle J_k - \sum_k \chi_{\mathit{ext}, k}^{\tau_{1}} \mathcal L_S(A) \\	
	&=  (\mathcal L - \mathcal L_S) ( \bar v_{\mathit{ext}}) + \sum_k [ \mathcal L_S, \chi_{\mathit{ext}, k}^{\tau_{1}}] (v_{\mathit{ext}, k}) +  \sum_k \chi_{\mathit{neck}, k}^{\tau_{2}} (\mathcal L - \mathcal L_N)( \bar u_{\mathit{neck}})   \\
	&\qquad - \sum_k \chi_{\mathit{ext}, k}^{\tau_{1}} \langle w_{\mathit{ext}, k},  J_k \rangle J_k - \sum_k \chi_{\mathit{ext}, k}^{\tau_{1}} \mathcal L_S(A)  + w
\end{align*}
since $\mathcal L (\bar u_{\mathit{neck}}) = ( \mathcal L - \mathcal L_N)(\bar u_{\mathit{neck}}) + \sum_k [\mathcal L_N, \chi_{\mathit{neck}, k}^{\tau_{4}}](u_{\mathit{neck}, k}) + \sum_k \chi_{\mathit{neck}, k}^{\tau_{4}} \mathcal L_N( u_{\mathit{neck}, k})$.  The facts $\chi_{\mathit{ext}, k}^{\tau_{1}} \chi_{\mathit{ext}, k}^{\tau_{2}} = \chi_{\mathit{ext}, k}^{\tau_{2}}$ and $ \chi_{\mathit{neck}, k}^{\tau_{2}}  \chi_{\mathit{neck}, k}^{\tau_{3}}  \chi_{\mathit{neck}, k}^{\tau_{4}} =  \chi_{\mathit{neck}, k}^{\tau_{2}}$ and $\chi_{\mathit{neck}, k}^{\tau_{4}} [ \mathcal L_N , \chi_{\mathit{neck}, k}^{\tau_{2}} ] = 0 $ have also been used in the calculations above.  The result now follows since both $\mathcal L - \mathcal L_S$ and $\mathcal L - \mathcal L_N$ can be handled using the estimate \eqref{eqn:lindifest} while 
$$| [ \mathcal L_S, \chi_{\mathit{ext}, k}^{\tau_{1}}] v_{\mathit{ext}, k} |_{\coan} \leq C |v_{\mathit{ext}, k}|_{\ctan(\mathrm{supp} (\nabla \chi_{\mathit{ext}, k}^{\tau_{1}}))} \leq C \left( \frac{\tau_{1}}{\tau_{2}} \right)^{2 - \nu} | w |_{\coan}$$
can be made as small as desired by adjusting the ratio $\tau_{1} / \tau_{2}$.  

The iteration leading to the exact solution of the equation $\mathcal L(u) = w + \mathcal E(w)$ now works as follows.  Using the steps above, for every $n \geq 0$ one has functions $u^{(n)} \in \ctan(\apsolf)$ and $w^{(n)} \in \coan(\apsolf)$ satisfying $\mathcal L(u^{(n)}) = w^{(n-1)} + \mathcal E( u^{(n-1)}) + w^{(n)}$ and $w^{(0)} = w$ along with the estimates
$$ |u^{(n)}|_{\ctan} + | \mathcal E(w^{(n-1)}) |_{C^{2, \alpha}_0} \leq C | w^{(n-1)} |_{\coan} \qquad \mbox{and} \qquad |w^{(n)}|_{\coan} \leq \frac{1}{2} | w^{(n-1)}|_{\coan}   \, .$$
Consequently the series $u := \sum_{n=1}^\infty (-1)^{n+1} u^{(n)}$ and $\mathcal E(w) := \sum_{n=1}^\infty (-1)^{n+1} \mathcal E (w^{(n-1)})$ converge in the appropriate norms and satisfy $\mathcal L (u) = w + \mathcal E (w)$ along with the desired estimates.
\end{proof}

The definition of the finite-dimensional image of the map $\mathcal E : \coan(\apsolf) \rightarrow \tilde{\mathcal W}^F$ is a by-product of Step 3 of the previous proof.

\begin{defn}
	Define 
	$$\tilde {\mathcal W}^F := \mathrm{span} \{  \chi_{\mathit{ext}, k}^{\tau_{1}} J_k \,  , \,  \chi_{\mathit{ext}, k}^{\tau_{1}} \mathcal L_S(\eta_k^+ \chi_{\mathit{neck}, k}^{\tau_{1}} ) : k = 1,  \ldots, K-1 \} \cup \{ \chi_{\mathit{ext}, K}^{\tau_{1}} J_K  \} \, . $$
\end{defn}

\subsubsection{The Non-Linear Estimates} 

The next task is to find estimates for the $C^{0,\alpha}_{\nu}$ norm of the quadratic remainder term 
$\mathcal Q $ and the error term $\mathcal H$ for the finite-length surface. We do this by
combining the estimates from Section \ref{sec:prelimcalc} with the specifics of the construction 
of $\apsolf$ from Section \ref{sec:approxsol}.  

\begin{lemma}
	\label{lemma:secondffsmall}
	Pick $x \in \apsolf$.  Then $x$ belongs to one of the normal coordinate charts used in the construction of $\apsolf$ where the second fundamental form with respect to the Euclidean metric is $\mathring B(x)$.  At this point, the estimate 
	$$r  \big( \| \mathring B(x) \| |f(x)| + \| \mathring\nabla f(x) \|\big)  \leq Cr^{\nu} |f|_{C^{2,\alpha}_\nu}$$
	holds for all $f \in C^{2, \alpha}_\ast (\apsolf)$, where $C$ is a constant independent of $r$, $\eps$ and $\delta$.
\end{lemma}

\begin{proof}
	Collecting the estimates for $\mathring B$ from Section \ref{sec:meancurvest}, one finds that 
	\begin{equation}
		\label{eqn:bfestimate}
		\zeta_r (x) \| \mathring B(x) \| + \zeta_r^2(x) \| \mathring \nabla \mathring B(x) \| \leq C 
	\end{equation}
	where $C$ is some constant independent of $r$, no matter where $x$ is located in $\apsolf$.  Consequently,
	\begin{align*}
		\| \mathring B(x) \| |f(x)| + \| \mathring \nabla f(x) \| &= \zeta_r^{\nu - 1}(x) \big( \| \mathring B(x) \| \zeta_r(x) \cdot \zeta_r^{-\nu}(x) |f(x)| + \zeta_r^{1-\nu}(x) \| \mathring \nabla f(x) \| \big)
	\end{align*}
which yields the desired estimate.
\end{proof}

\noindent Hence it is true that $r ( |f| \| \mathring B \| + \| \mathring \nabla f \|) \ll 1$ can be ensured by keeping $|f|_{C^{2, \alpha}_\nu}$ small enough.  This condition validates the computations of Section \ref{sec:prelimcalc}.  The following estimates are a consequence.

\begin{prop}
	\label{prop:nonlinest}
	There exists $M>0$ so that if  $f_1, f_2 \in \ctan(\apsolf)$ for $\nu \in (1, 2)$ and satisfying $|f_1|_{C^{2,\alpha}_\nu} + |f_2|_{C^{2,\alpha}_\nu} \leq M$, then the quadratic remainder term $\mathcal Q$ satisfies the estimate
	$$| \mathcal Q( f_1) - \mathcal Q( f_2) |_{\coan} \leq C r^{\nu - 1}  |f_1 - f_2 |_{\ctan} \max \big\{ |f_1|_{\ctan}, |f_2|_{\ctan} \big\}$$
	where $C$ is a constant independent of $r$, $\eps$ and $\delta$.
\end{prop}

\begin{proof}
Choose a point $x \in \apsolf$.  Then $x$ belongs to one of the normal coordinate charts used in the construction of $\apsolf$.  If $r( | f_i| \| \mathring B \|+\| \mathring \nabla f_i\|)$ is sufficiently small, it is possible to invoke Lemma \ref{lemma:newfnquadest} along with the estimate of the previous lemma and immediately deduce
\begin{align*}
		\zeta_r^{2-\nu} | \mathcal Q(f_1) -\mathcal Q(f_2) | & \leq C \zeta_r^{-\nu} | f_1 - f_2 | \cdot \max_i   \Big(|f_i| \| \mathring B \|^3 + \|\mathring \nabla f_i \| \| \mathring B \|^2 +  \|\mathring \nabla f_i \| \| \mathring \nabla \mathring B \|  \\
		&\qquad \qquad + \| \mathring \nabla^2 f_i \| \| \mathring B \| + | f_i |+ (1 + \| Y \|)  \| \mathring \nabla f_i \| \Big) \cdot \zeta_r^{2}\\
		&\qquad + C \zeta_r^{1-\nu} \| \mathring \nabla f_1 - \mathring \nabla f_2 \| \cdot \max_i \Big(  | f_i | \| \mathring B \|^2 + \|\mathring \nabla f_i \| \| \mathring B \| + | f_i | \| \mathring \nabla \mathring B \|   \\
		&\qquad \qquad + \| \mathring \nabla^2 f_i \| + (1 + \| Y \|)  | f_i |+  (1 + \| Y \|^2)  \| \mathring \nabla f_i \| +  \| \mathring \nabla f_i \| \| \mathring \nabla^2 f_i \| \Big) \cdot \zeta_r \\
		&\qquad + C \zeta_r^{2-\nu} \| \mathring \nabla^2 f_1 - \mathring \nabla^2 f_2 \| \cdot \max_i \big( |f_i| \| \mathring B \| + \| \mathring \nabla f_i \|^2 \big) \\
		&\leq C \big( r^{\nu - 1} + r^{2\nu - 2} + r^{\nu +1} + r^{\nu +2} \big) | f_1 - f_2 |_{\ctan}  \max_i \big\{ |f_i|_{\ctan} \big\} 
\end{align*}
since $1/C \leq \| Y \| \zeta_r^{-1} \leq C$ for some universal constant $C$.  The desired estimate follows.
\end{proof}

\begin{prop}
	\label{prop:nonlinesttwo}
	There exists $M>0$ so that if  $f_1, f_2 \in \ctan(\apsolf)$ for $\nu \in (1, 2)$ and satisfying $|f_1|_{C^{2,\alpha}_\nu} + |f_2|_{C^{2,\alpha}_\nu} \leq M$, then the error term $\mathcal H$ satisfies the estimate
	$$| \mathcal H( f_1) - \mathcal H( f_2) |_{\coan} \leq C r^{4} |f_1 - f_2 |_{\ctan} $$
	where $C$ is a constant independent of $r$, $\eps$ and $\delta$.
\end{prop}

\begin{proof}
	Similar computations.
\end{proof}

\subsubsection{The Fixed-Point Argument}
We are now in a position to solve the CMC equation up to a finite-dimensional error.  
Let $E :=  H \big[\apsolf \big]  - \frac{2}{r} $ and $R(r, \eps, \delta) := \max \big\{ r^{3 - \nu}  , \, r^{5-\nu} \eps^{1/2 - 3\nu/4} , \, r^{1-\nu} \eps^{3/2 - 3\nu/4}, \delta r \eps^{1 - 3\nu/4} \big\}$.   Additionally, assume $r^3 < \eps < r^2 \ll 1$ and $\delta < \eps^{1/2}$.  All of this is justified \emph{a posteriori}.  The following estimates have been established.
\begin{itemize}
	\item The mean curvature satisfies $|E|_{\coan} \leq C R(r, \eps, \delta)$.
	
	\item There is a parametrix $\mathcal R$ satisfying $\mathcal L \circ \mathcal R = \mathit{id} - \mathcal E$ where $\mathcal E$ maps into the finite-dimensional space $\tilde{\mathcal W}$ and  $| \mathcal R(w) |_{\ctan} + | \mathcal E(w)|_{C^{2, \alpha}_0} \leq C |w|_{\coan}$ for all $w \in \coan(\apsol)$.
		
	\item The quadratic remainder satisfies $| \mathcal Q(f_1) - \mathcal Q(f_2) |_{\coan} \leq C r^{\nu-1} |f_1 - f_2|_{\ctan} \max_i \big\{  |f_i|_{\ctan} \big\}$ for all $f_1, f_2 \in \ctan(\apsolf)$ with sufficiently small $\ctan$ norm.
	
	\item The error term satisfies $| \mathcal H(f_1) - \mathcal H(f_2) |_{\coan} \leq C r^4 |f_1 - f_2|_{\ctan} $ for all $f_1, f_2 \in \ctan(\apsolf)$ with sufficiently small $\ctan$ norm.

\end{itemize}

\noindent One can now assert the following.

\begin{prop}
	\label{prop:soluptocoker}
	There exists $w:= w_r(\sigma, \delta) \in \coan(\apsolf)$ and corresponding $f:= f_r(\sigma,\delta)  \in \ctan(\apsolf)$ defined by $f := \frac{1}{r} \mathcal R \big( w - H \big[\apsolf \big]  + \frac{2}{r} \big)$ so that 
	\begin{equation}
		\label{eqn:uptofindim}
		H \big[ \mu_{rf} \big( \apsolf \big) \big]  - \frac{2}{r} = - \mathcal E 
	\end{equation}
	where $\mathcal E \in \tilde{\mathcal W}^{\mathit{F}}$. The estimate $|f|_{\ctan} \leq Cr^{-1} R(r, \eps, \delta)$ holds for the function $f$, where the constant $C$ is independent of $r$, $\eps$ and $\delta$.  Finally, the mapping $(\sigma, \delta) \mapsto f(\sigma, \delta)$ is smooth in the sense of Banach spaces.
\end{prop}

\begin{proof}
By the last three bullet points above, the map $w \mapsto \mathcal N_r(w)$ satisfies 
\begin{align*}
	|\mathcal N_r(w_1) - \mathcal N_r(w_2) |_{\coan} &\leq C \big( r^{\nu - 1} R(r, \eps, \delta) +r^4 \big) |w_1 - w_2 |_{\coan} 
\end{align*}
where $C$ is independent of $r$ and $\eps$.  Since $r^{\nu -1 } R(r, \eps, \delta)$ and $r^4$ can be made as small as desired by a sufficiently small choice of $r$ and $\eps$ with $r^3 \leq \eps \leq r^2$ and $\delta \leq \eps^{1/2}$, it is thus true that $\mathcal N_r$ is a contraction mapping on the ball of radius $R(r, \eps, \delta)$ for such $r$, $\eps$ and $\delta$.  Hence a solution of \eqref{eqn:uptofindim} satisfying the desired estimate can be found.  The dependence of this solution on the parameters $(\sigma, \delta)$ is smooth as a natural consequence of the fixed-point process.
\end{proof}

\subsection{The One-Ended Surface}

\subsubsection{Strategy}

The strategy for solving the CMC equation \eqref{eqn:tosolve} in the case of the one-ended surface $\apsoloe$ must be modified in order to take the non-compactness of $\apsoloe$ into account.  In fact, the modification required can be understood by considering the outcome of the linear analysis, specifically the nature of the parametrix for $\mathcal L$.  In this case, the outcome of the construction the parametrix, which will mimic Proposition \ref{prop:linest} as closely as possible, will  be a parametrix $\mathcal R : \coann(\apsoloe)\rightarrow \ctann (\apsoloe) \oplus \tilde{\mathcal V}$ satisfying $\mathcal L \circ \mathcal R = \mathit{id} + \mathcal E$.  The operator $\mathcal E $ again maps into a finite-dimensional subspace $\tilde{\mathcal W}^{\mathit{OE}}$.  The subspace $\tilde{\mathcal V}$ is the new ingredient, and can be explained as follows.  First, let $J_{\mathit{Del}}^s$ for $s= 0, 1$ be the bounded and linearly growing Jacobi fields of the standard Delaunay surface and define the space 
$$\tilde {\mathcal V} := \mathrm{span} \{ \chi_{\mathit{Del}}^{\bar \tau} J_{\mathit{Del}}^1 \, , \, \chi_{\mathit{Del}}^{\bar \tau} J_{\mathit{Del}}^2 \}$$
where $\chi_{\mathit{Del}}^{\bar \tau}$ is a smooth, monotone cut-off function that transitions from zero to one in the neck region where the Delaunay end of $\apsoloe$ is attached to the finite part of $\apsoloe$.  The reason $\tilde {\mathcal V}$ is needed is simply because $\mathcal L : \ctann(\apsoloe) \rightarrow \coann(\apsoloe)$ is not surjective but becomes so when growth like the first non-decaying Jacobi fields of $\mathcal L$ is permitted.  But now the fact that one component of the solution of the linearized problem does not decay forces the modified approach that will be outlined in the next two paragraphs, since the quadratic remainder of the mean curvature will not behave appropriately for this component.  An approach similar to the one proposed below has been used in \cite{ratzkinthesis}.

To compensate for the non-decaying component of the solution of the linearized equation, one proceeds as follows.  Let $\mathcal R^{(1)}$ denote the component of $\mathcal R$ mapping into $\ctann (\apsoloe)$ and let $\mathcal R^{(2)}$ be the component of $\mathcal R$ mapping into $\tilde{\mathcal V}$.  Furthermore, if $\mathcal R^{(2)} (w) = a_1 \chi_{\mathit{Del}}^{\bar \tau} J_{\mathit{Del}}^1  + a_2 \chi_{\mathit{Del}}^{\bar \tau} J_{\mathit{Del}}^2$, then $\mathcal R^{(2)} (w) := (a_1(w), a_2(w))$ despite the slight abuse of notation that this represents.  Now the equation that needs to be solved is still 
\begin{equation}
	\label{eqn:tosolveoe}
	H \big[ \mu_{rf} \big( \apsoloe \big) \big] = \frac{2}{r}
\end{equation}
but for $f \in \coann(\apsoloe)$.  Recall that the last two free parameters of $\apsoloe$, namely $\sigma_K$ and $\delta_K$, parametrize asymptotically non-trivial deformations of $\apsoloe$.  Namely, these cause the period and location of the entire Delaunay end to change.  The idea for converting \eqref{eqn:tosolveoe} into a fixed-point problem that can be solved in the standard way, is to associate $\mathcal R^{(1)}$ with $f$ and $\mathcal R^{(2)}$ with the parameters $(\sigma_K, \delta_K)$ in an appropriate way.

This can be done as follows.  Recall that there are specific values $\sigma_K = \mathring \sigma_K$ and $\delta_K = 0$ which produce optimal matching in the assembly of $\apsoloe$.  With slight abuse of notation, write $\tilde \Sigma_r^{\mathit{OE}}(a_1, a_2) := \apsoloe$ with $(\sigma_K, \delta_K ) = (\mathring \sigma_K + a_1, a_2)$.  Given the flexibility one has in the choice of $\tilde{\mathcal V}$, one can arrange that 
$$\left. \frac{\partial}{\partial a_i} H \big[ \tilde \Sigma_r^{\mathit{OE}}(a_1, a_2) \big] \right|_{a = 0} = \mathcal L ( \chi_{\mathit{Del}}^{\bar \tau} J_{\mathit{Del}}^i )$$
for $i = 1, 2$.   Consequently, the \emph{Ans\"atze} 
\begin{align*}
	f &:= \frac{1}{r} \mathcal R^{(1)} \big(w -  H \big[ \mu_{rf} \big( \apsoloe \big) \big] + \tfrac{2}{r} \big) \\
	(a_1 , a_2) &:= \frac{1}{r} \mathcal R^{(2)} \big(w -  H \big[ \mu_{rf} \big( \apsoloe \big) \big] + \tfrac{2}{r} \big)
\end{align*} 
along with the expansion of the mean curvature into its constant, linear and quadratic and higher parts transforms the equation \eqref{eqn:tosolveoe} into the fixed-point problem
\begin{align}
	\label{eqn:cmapoe}
	w &= - \mathcal Q_{a} \circ \mathcal R^{(1)} \left(w -  H \big[ \mu_{rf} \big( \apsoloe \big) \big] + \frac{2}{r} \right) - \mathcal H \circ \mathcal R \left(w -  H \big[ \mu_{rf} \big( \apsoloe \big) \big] + \frac{2}{r} \right)
\end{align}
up to a term in $\tilde{\mathcal W}^{\mathit{OE}}$, where $\mathcal Q_a$ denotes the quadratic remainder of the mean curvature of $\tilde \Sigma_r^{\mathit{OE}} (a_1, a_2)$.  One must now show that the mapping  $\mathcal N_r : \coann(\apsoloe) \rightarrow \apsoloe$ given by the right hand side of \eqref{eqn:cmapoe} is a contraction mapping onto a neighbourhood of zero containing $H \big[\apsoloe \big]  - \frac{2}{r}$.  If so, then one has solved the equation \eqref{eqn:tosolveoe} up to a term in $\tilde{\mathcal W}^{\mathit{OE}}$.  This finite-dimensional error term must of course still be dealt with in order to find a true CMC surface near to $\apsoloe$.  This will also be carried out in Section \ref{sec:balancing}.

\subsubsection{The Linear Analysis}
\label{sec:linanalysisoe}

To begin the construction of the parametrix in the case of the one-ended surface, one must first define an additional set of partitions of unity for $\apsoloe$ as follows. Denote $\mathit{Cyl}_{\bar \tau} := \{ (x, t) \in M : \| x \| \leq r \; \; \mbox{and} \; \; t \geq \bar \tau \}$ for any $\bar \tau \in \R$ and then define $ \mathcal D^{\bar \tau} := \apsoloe \cap \mathit{Cyl}_{\bar \tau}$ as well as a smooth, monotone cut-off function $\chi^{\bar \tau}_{\mathit{Del}}$ that equals one in $\mathcal D^{\bar \tau}$ and vanishes in $\apsoloe \setminus \mathcal D^{\bar \tau - r}$. 

A second important  ingredient that will be used in the construction of the parametrix is a careful analysis of the properties of the linearized mean curvature operator of a near-degenerate Delaunay surface with respect to the $C^{k, \alpha}_{\nu, \bar \nu}$ norm.  This was carried out in \cite{mazzeopacardends} and the relevant results from \cite[\S 4]{mazzeopacardends} can be adapted to the needs of this paper and will be quoted whenever they are used in the proof given below.

\begin{prop}
	\label{prop:linestoe}
	Let $(\nu, \bar \nu) \in (1, 2) \times (-1, 0)$.  There is an operator $\mathcal R: \coann(\apsoloe) \rightarrow \ctann(\apsoloe) \oplus \tilde{\mathcal V} $ that satisfies $\mathcal L \circ \mathcal R = \mathit{id} - \mathcal E$ where $\mathcal E : \coann(\apsoloe) \rightarrow  \tilde{\mathcal W}$.  Here $\tilde{\mathcal W}^{\mathit{OE}}$ is a  finite-dimensional space that will be defined below.  The estimates satisfied by $\mathcal R$ and $\mathcal E$ are
	$$| \mathcal R(w) |_{\ctann \oplus \tilde{\mathcal V}} + |\mathcal E(w)|_{C^{2,\alpha}_{0, \bar \nu}} \leq C | w |_{\coann}$$
	for all $w \in \coan(\apsoloe)$, where $C$ is a constant independent of $r$, $\eps$, $\delta$ and $K$. 
\end{prop}

\begin{proof}
Let $w \in \coann(\apsoloe)$ be given.  The solution of the equation $\mathcal L(u) = w + \mathcal E(w)$ will be constructed broadly along the same lines as in Proposition \ref{prop:linest} in that local solutions on the spherical constituents, the necks and the Delaunay end of $\apsoloe$ will be patched together.  In this case, however, a preliminary step is needed to reduce the interaction between the Delaunay end and the finite part of $\apsoloe$.  This amounts to showing that one can assume that $w \equiv 0$ in a large part of $\apsoloe$.

\paragraph*{Step 0.}  Define $w_{\mathit{mid}} := \chi^{t_{K}}_{\mathit{Del}} \, (1 - \chi^{t_{2K}}_{\mathit{Del}}) \, w$.  Using the methods of Proposition \ref{prop:linest}, one can solve the Dirichlet problem $\mathcal L (u_{\mathit{mid}}) = w_{\mathit{mid}} + \mathcal E(w_{\mathit{mid}})$ and $u_{\mathit{mid}} = 0$ on $\partial \big[ \mathcal D^{t_{K}} \setminus \mathcal D^{t_{2K}} \big]$. The estimate one obtains is $|u_{\mathit{mid}}|_{\ctan} + | \mathcal E(w_{\mathit{mid}}) |_{C^{2, \alpha}_0} \leq C |w_{\mathit{mid}}|_{\coan} \leq C |w|_{\coann}$ where $C$ is independent of $r$, $\eps$ and $K$.  Then $u_{\mathit{mid}}$ can be extended to all of $\apsoloe$ by defining $\bar u_{\mathit{mid}} :=  \chi^{t_{K}-r}_{\mathit{Del}} \, (1 - \chi^{t_{2K}+r}_{\mathit{Del}}) \, u_{\mathit{mid}}$.    If one now solves $\mathcal L(u) = w - \mathcal L (\bar u_{\mathit{mid}})$, then the function $u + u_{\mathit{mid}}$ solves $\mathcal L(u + u_{\mathit{mid}}) = w$.  The advantage will be that the function $w - \mathcal L (\bar u_{\mathit{mid}})$ vanishes in $\mathcal D^{t_{K}} \setminus \mathcal D^{t_{2K}}$ but still satisfies the same estimates as did $w$.  One can also assume that $K$ can be as large as desired.  

\paragraph*{Step 1.} Let $w \in \coann(\apsoloe)$ be given and assume that $w \equiv 0$ in $\mathcal D^{t_{K}} \setminus \mathcal D^{t_{2K}} $.  Consider the equation $\mathcal L(u_{\mathit{fin}}) = w_{\mathit{fin}} + \mathcal E(w_{\mathit{fin}})$ where $w_{\mathit{fin}} := (1 - \chi_{\mathit{Del}}^{t_{2K}}) \, w$ and view the function on the right hand side as being defined on the finite-length surface.  Using the methods of Proposition \ref{prop:linest}, one can find a solution $u_{\mathit{fin}} \in \ctan (\apsolf)$ satisfying $|u_{\mathit{fin}}|_{\ctan} + \mathcal E(w_{\mathit{fin}}) |_{C^{0,\alpha}_0} \leq C | w |_{\coann}$ for some constant $C$ independent of $r$, $\eps$ and $K$.  Note that one must view $\mathcal D^{t_{K}} \setminus \mathcal D^{t_{2K}} $ as being very close to a union of spherical and neck regions in order to do so.   Furthermore, by carefully analyzing the iteration process that leads to the solution, the fact that $\mathit{supp} (w_{\mathit{fin}}) \subseteq \apsoloe \setminus \mathcal D^{t_{K}}$ implies that for $x \in \apsoloe$ near $\partial \mathcal D^{t_{2K}}$ one has $\zeta_r^{-\nu}(x) |u_{\mathit{fin}}(x)| \leq C \me^{-K} |w|_{\coann}$ for some constant $C$ independent of $r$, $\eps$ and $K$.  Extend $u_{\mathit{fin}}$ to all of $\apsoloe$ by defining $\bar u_{\mathit{fin}} := (1 - \chi^{t_{2K} + r}_{\mathit{Del}}) \, u_{\mathit{fin}}$.

\paragraph*{Step 2.} For $w \in \coann(\apsoloe)$ from Step 1, set $w_{\mathit{Del}} := \chi_{\mathit{Del}}^{t_{2K}} \, w$ and consider now the equation $\mathcal L_D (u_{\mathit{Del}}) = w_{\mathit{Del}}$ but viewed as an equation on the complete Delaunay surface $D_r$.   Then using the methods of \cite[\S 4]{mazzeopacardends} and \cite{pacardnotes} along with the condition $\bar \nu \in (-1, 0)$, there is a solution of this equation which can be decomposed as $u_{\mathit{Del}} = v_{\mathit{Del}} + a^1_{\mathit{Del}} \, \chi_{\mathit{Del}}^{\bar \tau} J_{\mathit{Del}}^1 + a_{\mathit{Del}}^2 \, \chi_{\mathit{Del}}^{\bar \tau} J_{\mathit{Del}}^2$.  Here $v_{\mathit{Del}} \in \ctann(D_r)$ and one has the estimate $|v_{\mathit{Del}}|_{\ctan} + |a_{\mathit{Del}}^1| + |a_{\mathit{Del}}^2| \leq C |w_{\mathit{Del}}|_{\coann}$ for some constant independent of $r$ and $\eps$.  Furthermore, because $\mathit{supp} (w_{\mathit{Del}}) \subseteq \apsoloe \cap \mathcal D^{t_{K}}$ one can arrange to have  $\zeta_r^{-\nu}(x) |u_{\mathit{fin}}(x)| \leq C \me^{-K} |w|_{\coann}$ for $x \in \apsoloe$ near $\partial \mathcal D^{t_{2K}}$ for some constant $C$ independent of $r$, $\eps$ and $K$.   Now it is possible to view $\apsoloe \cap \mathcal D^{t_{K}}$ as a graph over $D_r \cap  \mathcal D^{t_{K}}$ and hence $u_{\mathit{Del}}$ can be viewed as a function defined on $\apsoloe \cap \mathcal D^{t_{K}}$.  Extend this function to all of $\apsoloe$ by defining $\bar u_{\mathit{Del}} := \chi_{\mathit{Del}}^{t_{K} - r} \, u_{\mathit{Del}}$.
 
\paragraph*{Step 3.} The estimate of $\mathcal L(u) - w$ up to a finite-dimensional error term proceeds in the same way as in Step 3 of Proposition \ref{prop:linest}.  The extra exponential decay of $\bar u_{\mathit{fin}}$ and $\bar u_{\mathit{Del}}$ near $\partial \big[ \mathcal D^{t_{K}} \setminus \mathcal D^{t_{2K}} \big]$ ensures that the cut-off errors that arise there are small.  Consequently one can iterate the steps above and find the desired solution $\mathcal R(w) \in \ctann(\apsoloe) \oplus \tilde{\mathcal V}$ and satisfying the desired estimate.
\end{proof}

The definition of the finite-dimensional image of the map $\mathcal E : \coan(\apsoloe) \rightarrow \tilde{\mathcal W}^{\mathit{OE}}$ is once again a by-product of the previous proof.

\begin{defn}
	Define
	$$\tilde {\mathcal W}^{\mathit{OE}} := \mathrm{span} \{  \chi_{\mathit{ext}, k}^{\tau_{1}} J_k \,  , \,  \chi_{\mathit{ext}, k}^{\tau_{1}} \mathcal L_S(\eta_k^+ \chi_{\mathit{neck}, k}^{\tau_{1}} ) : k = 0,  \ldots, K-1 \} \cup \{ \chi_{\mathit{ext}, K}^{\tau_{1}} J_K \}  \, .$$
\end{defn}

\subsubsection{The Non-Linear Estimates}

The estimates for the $\coann$ norm of the quadratic remainder term $\mathcal Q_a $ in the case of the one-ended surface are very similar to the analogous estimates for the finite-length surface.   First, the result of Lemma \ref{lemma:secondffsmall} continues to hold because the calculations are essentially identical, the only difference being the need to multiply by factors of $\me^{-\bar \nu T}$ along the end of $\apsoloe$.  These growing factors are compensated for by the exponential decay assumed for the function $f$.  Consequently it is possible to make $\| B \| |f| + \| \nabla f \|$ pointwise small everywhere by choosing $|f|_{\ctann}$ sufficiently small.  

Next, the non-linear estimate analogous to Proposition \ref{prop:nonlinest} follows similarly because the terms in $\mathcal Q_a$ and $\mathcal H$ coming from the background metric decay exponentially.  One has the following results.

\begin{prop}
	\label{prop:nonlinestoe}
	There exists $M>0$ so that if  $f_1, f_2 \in \ctann(\apsoloe)$ for $\nu \in (1, 2)$ and $\bar \nu \in (-1, 0)$ sufficiently close to zero and satisfying $|f_1|_{\ctann} + |f_2|_{\ctann} \leq M$, then the quadratic remainder term $\mathcal Q_a$ satisfies the estimate
	$$| \mathcal Q_a( f_1) - \mathcal Q_a( f_2) |_{\coann} \leq r^{\nu - 1} C |f_1 - f_2 |_{\ctann} \max_i \big\{ |f_i|_{\ctann}  \big\}$$
	where $C$ is a constant independent of $r$, $\eps$ and $\delta$. 
\end{prop}

\begin{prop}
	\label{prop:nonlinesttwooe}
	There exists $M>0$ so that if  $f_1, f_2 \in \ctann(\apsolf)$ for $\nu \in (1, 2)$ and $\bar \nu \in (-1, 0) $ sufficiently close to zero and satisfying $|f_1|_{C^{2,\alpha}_\nu} + |f_2|_{C^{2,\alpha}_\nu} \leq M$, then the error term $\mathcal H$ satisfies the estimate
	$$| \mathcal H( f_1) - \mathcal H( f_2) |_{\coann} \leq C r^{4} |f_1 - f_2 |_{\ctann} $$
	where $C$ is a constant independent of $r$, $\eps$ and $\delta$.
\end{prop}

\subsubsection{The Fixed-Point Argument}

The fixed-point argument in the case of the one-ended surface is again broadly similar to the argument in the case of the finite-length surface.  However, the strategy adopted for dealing with the non-decaying component of the parametrix requires 
some additional care.  As before, let $E :=  H \big[\apsoloe \big]  - \frac{2}{r} $ and $R(r, \eps, \delta) := \max \big\{ r^{3 - \nu}  , \, r^{5-\nu} \eps^{1/2 - 3\nu/4} , \, r^{1-\nu} \eps^{3/2 - 3\nu/4}, \delta r^{1-\nu} \eps^{1-3\nu/4} \big\}$.  Additionally, assume $r^3 < \eps < r^2 \ll 1$ and $\delta < \eps^{1/2}$ as before.  The following have been established.
\begin{itemize}
	\item The mean curvature satisfies $|E|_{\coann} \leq C R(r, \eps, \delta)$.
	
	\item There is a parametrix $\mathcal R$ satisfying $\mathcal L \circ \mathcal R = \mathit{id} - \mathcal E$ where $\mathcal E$ maps into the finite-dimensional space $\tilde{\mathcal W}$ and $| \mathcal E(w)|_{C^{2, \alpha}_0} \leq C |w|_{\coann}$ for all $w \in \coann(\apsoloe)$.  But now $\mathcal R$ decomposes as $\mathcal R^{(1)} + \mathcal R^{(2)}$ and $| \mathcal R^{(1)}(w) |_{\ctann} + \| \mathcal R^{(2)}(w) \|_{\tilde{\mathcal V}}  \leq C |w|_{\coann}$ for all $w \in \coann(\apsoloe)$.
		
	\item The quadratic remainder satisfies $| \mathcal Q_a(f_1) - \mathcal Q_a(f_2) |_{\coann} \leq C r^{\nu-1} |f_1 - f_2|_{\ctann} \max_i \big\{  |f_i|_{\ctann} \big\}$ for all $f_1, f_2 \in \ctann(\apsoloe)$ with sufficiently small $\ctann$ norm.
	
	\item The error term satisfies $| \mathcal H(f_1) - \mathcal H(f_2) |_{\coann} \leq C r^4 |f_1 - f_2|_{\ctann} $ for all $f_1, f_2 \in \ctann(\apsolf)$ with sufficiently small $\ctann$ norm.
	
\end{itemize}

\noindent One can now assert the following.  It's proof is entirely analogous to the proof of Proposition \ref{prop:soluptocoker}.

\begin{prop}
	\label{prop:soluptocokeroe}
	There exists $w:= w_r(\sigma, \delta) \in \coann(\apsoloe)$, corresponding $f:= f_r(\sigma, \delta)  \in \ctann(\apsoloe)$ and $(a_1, a_2) \in \tilde {\mathcal V}$  defined by $f := \frac{1}{r} \mathcal R^{(1)} \big( w - H \big[\apsoloe \big]  + \frac{2}{r} \big)$ and $(a_1, a_2) :=  \mathcal R^{(2)} \big( w - H \big[\apsoloe \big]  + \frac{2}{r} \big)$ so that 
	\begin{equation*}
		H \big[ \mu_{rf} \big( \tilde \Sigma_r^{\mathit{OE}} (a_1, a_2) \big) \big]  - \frac{2}{r} = - \mathcal E
	\end{equation*}
	where $\mathcal E \in \tilde{\mathcal W}^{\mathit{OE}}$.  The estimate $|f|_{\ctann} \leq Cr^{-1} R(r, \eps, \delta)$ holds for the function $f$, where the constant $C$ is independent of $r$, $\eps$ and $\delta$.   Finally, the mapping $(\sigma, \delta) \mapsto f(\sigma, \delta)$ is smooth in the sense of Banach spaces.

\end{prop}

\section{Solving the Finite-Dimensional Problem}
\label{sec:balancing}

\subsection{Strategy}
It has now been established for both families of surfaces that if $r$, $\sigma$ and $\delta$ are sufficiently 
small, one can find $w_r(\sigma, \delta)  \in C^{0, \alpha}_\ast(\apsol)$ and corresponding 
$f_r(\sigma, \delta) \in C^{2, \alpha}_\ast(\apsol)$ so that 
$$H \big[ \mu_{f_r(\sigma, \delta)} \big(\apsol \big) \big]  - \frac{2}{r}  = \mathcal E_r (\sigma, \delta) $$
where $\mathcal E_r (\sigma, \delta)$ is an error term belonging to the finite-dimensional space 
$\tilde{\mathcal W}^\ast$ and whose dependence on the free parameters in $\apsol$ has been indicated explicitly. 
To complete the proof of the main theorem, we must show that it is possible to find a solution where these 
error terms vanish identically.  

Consider the \emph{balancing map} defined by
\begin{equation}
	\label{eqn:balmapdef}
	B_r(\sigma, \delta) := \pi \big( \mathcal E_r (\sigma, \delta)  \big)
\end{equation}
where $\pi :  \tilde{\mathcal W}^\ast \rightarrow \R^{d}$ is a suitable bounded projection operator, where $d$ is the dimension of $\tilde {\mathcal W}^\ast$. (The operator $\pi$ will be a certain bijective $L^2$-orthogonal projection onto a finite-dimensional vector space.)  Note that $B_r$ is a smooth map between finite-dimensional vector spaces by virtue of the fact that the dependence of the solution $f_r(\sigma, \delta)$ on $(\sigma, \delta)$ is smooth and the mean curvature operator is a smooth map of the Banach spaces upon which it is defined.   It will be shown using the implicit function theorem for finite-dimensional maps that for every sufficiently small $r>0$, there exists $(\sigma, \delta ) := (\sigma(r), \delta(r))$ for which $B_r(\sigma, \delta) \equiv 0$ identically.  It is at this stage that the precise nature of the scalar curvature of the background manifold $M$ enters the picture: the behaviour of the scalar curvature along the geodesic $\gamma$ enters into the selection of the parameters $\sigma$ and $\delta$ to first approximation.

\subsection{The Balancing Formula}

The projection operators that will be used to study the finite-dimensional error $\mathcal E (w_r (\sigma, \delta))$ in the case of the finite-length surface and in the case of the one-ended surface can be defined as follows.  For $k = 0, \ldots, K$ let $J_k : S_k \rightarrow \R$ have its usual meaning; and let $I_k : r \eps_k N_k \rightarrow \R$ be the function defined by $I_k(x) := \|x \| (\| x \|^2 - \eps_k^2)^{-1/2}$ using the coordinates of the neck introduced in Section \ref{sec:approxsol}.  This latter function is in the kernel of the linearized mean curvature operator of the catenoid with respect to the Euclidean background metric; it is an odd function with respect to the center of the catenoid and is asymptotic to $\pm 1$.  Now for convenience let  $f := f_r(\sigma, \delta)$ and $\Sigma_f := \mu_{rf} (\apsol)$ denote the solution found in the previous section and define $\pi_k : \tilde{\mathcal W}^\ast \rightarrow \R^{2}$ by
$$ \pi_k (e) := \left( \int_{\Sigma_f} e \cdot \chi_{\mathit{neck}, k}^{\tau_{1}} I_k \, \dif \vol_g \, , \int_{\Sigma_f} e \cdot \chi_{\mathit{ext}, k}^{\tau_{4}} J_k \, \dif \vol_g   \right) \, .
$$
The notation for the cut-off functions from Section \ref{sec:linanalysis} has been used here.  A consequence of the following lemma is that if $\pi(e) = 0$ then $e=0$.  The proof is a straightforward computation.

\begin{lemma}
	\label{lemma:projop}
	Choose $e \in \tilde{\mathcal W}$ and write $e = \sum_{k=1}^K \big( a_k \chi_{\mathit{ext}, k}^{\tau_{1}} J_k + b_k \chi_{\mathit{ext}, k}^{\tau_{1}} \mathcal L_S (\eta_k^+ \chi_{\mathit{neck}, k}^{\tau_{1}} ) \big)$ for some $a_k, b_k \in \R$.  Then 
	$$ \pi_k(e) = \big( C_1 b_k + C_1' r^2 (\eps_k^{3/2} a_k - \eps_{k+1}^{3/2} a_{k+1}) \, , \, C_2 r^2 a_k  \big) $$
	where $C_1, C_1', C_2$ are constants independent of $r$, $\eps$ and $\delta$. 
\end{lemma}

A fundamental application of the expansions of the mean curvature found in Proposition \ref{prop:expansions} and equation \eqref{eqn:geomcexp} from Section \ref{sec:gnormcoord} is to derive a formula relating $\pi( \mathcal E (w_r(\sigma, \delta))$ to the geometry of $\apsol$.  It is via this formula that the location of the spheres and necks in $\apsol$ and the background geometry of $M$ conspire to determine when a nearby CMC surface can be found.

\begin{prop}
	\label{prop:integralexp}
	Let $\Sigma_f := \mu_{f} (\Sigma_r(\sigma, \delta))$ and $f := f_r(\sigma, \delta)$ for convenience.  Then the mean curvature of $\Sigma_{f}$ satisfies the formul\ae\
	\begin{subequations}
	\label{eqn:balform}
	\begin{align}
		&\begin{aligned}
			\label{eqn:neckbalformula}
			\int_{\Sigma_{f}} \!\! \left( H[ \Sigma_f] - \tfrac{2}{r} \right) \chi_{\mathit{neck}, k}^{\tau_{1}} I_k \, \dif \vol_g &=  C_0 \delta_k r  \eps^{3/2}  \\[-1ex]
			&\qquad + \mathcal O(\eps^{11/4} r^2) + \mathcal O( r^2 \eps^{3/2}) +  \mathcal O(  r^\nu R(r, \eps, \delta) \eps^{2})
		\end{aligned} \\[1ex]
		&\,\,\begin{aligned}
			\label{eqn:sphbalformula}
			\int_{\Sigma_{f}} \!\! \left( H[\Sigma_f] - \tfrac{2}{r} \right) \chi_{\mathit{ext}, k}^{\tau_{4}} J_k \, \dif \vol_g &= r  \big( C_1 \eps_{k+1} + C_1' \eps_{k+1}^{3/2} \big) - r \big( C_1 \eps_{k} + C_1' \eps_{k}^{3/2} \big) -C_2 r^4 \dot S(p_k) \\[-1ex]
			&\qquad + \mathcal O(\eps r^3) + \mathcal O( r^2 \eps^{3/2}) +  \mathcal O(  r^2 R(r, \eps, \delta))
		\end{aligned}
	\end{align}
	\end{subequations}
where $C_1, C_1', C_2$ are constants independent of $r$, $\eps$ and $\delta$; and $\dot S (p_k) := g ( \nabla_{\dot \gamma(t_k)} S(p_k), \dot \gamma (t_k) )$ is the component of the gradient of the scalar curvature of the ambient metric along the geodesic $\gamma$ at $p_k$.
\end{prop}

\begin{proof}
	The formula \eqref{eqn:sphbalformula} will be derived first.  Consider the subset $\mathcal V_k$ of $\Sigma_0$ consisting of the $k^{\mathit{th}}$ spherical region of $\Sigma := \apsol$ and its adjoining transition regions.  Let $X = \frac{\partial}{\partial t}$ and $\tilde X = \chi_{\mathit{ext}, k}^{\tau_{4}} X$ using the normal coordinate system centered at $p_k$ used in the definition of $\Sigma \cap \mathcal V_k$.  Furthermore, define the domain $\mathcal W_k$ in $M$ by requiring $\partial \mathcal W_k =  \mathcal V_k \cup c_1 \cup c_2$ where $c_1$ and $c_2$ are small embedded two-dimensional disks with boundaries $\partial c_1$ and $\partial c_2$ contained in $t=\mathit{constant}$ planes with $\mathring g$-conormal vectors $\mathring \nu_1$ and $\mathring \nu_2$ tangent to $\mathcal V_k$.

Now let $\Sigma'$ be either $\Sigma_f \cap \mathcal V_k$ or $\Sigma \cap \mathcal V_k$; let $X'$ be a vector field supported on this surface; and let $g'$ be any choice of background metric.  Define $I(\Sigma', X', g')$ to be the integral in \eqref{eqn:sphbalformula} except with $\Sigma_f$ replaced by $\Sigma'$ and $\tilde J_k$ replaced by $g'(X', N')$ where $N'$ is the $g'$-unit normal vector field of $\Sigma'$ and the mean curvature and volume form calculated from $g'$.  It is now simple to phrase the means by which the formula \eqref{eqn:sphbalformula} will be found.  First, $I(\Sigma_f, \tilde X, g)$ can be expressed as
\begin{equation}
	\label{eqn:balmapzero}
	\begin{aligned}
		I(\Sigma_f, \tilde X, g) &= I ( \Sigma, X, \mathring g) + \big(  I ( \Sigma, X, g)  -  I ( \Sigma, X, \mathring g) \big) \\
		&\qquad + \big(  I ( \Sigma_f, \tilde X,  g)  -  I ( \Sigma, \tilde X, g)  \big)+ \big( I ( \Sigma, \tilde X, g)  -  I ( \Sigma, X, g)  \big)  \, .
	\end{aligned}
\end{equation}
Then one can apply the first variation formula in Euclidean space to the first term, yielding a pair of boundary integrals; one can apply the expansions for the mean curvature with respect to the perturbed background metric from Lemma \ref{prop:expansions} to the second term, yielding a curvature quantity; and one can treat the third and fourth terms as small errors.

The details of the computation outlined above are as follows.  For the first term, the classical first variation formula for a surface with boundary in Euclidean space gives
\begin{align*}
	\int_{\Sigma \cap \mathcal V_k} \big( \mathring H[\Sigma] - \tfrac{2}{r} \big) J_k \, \dif \vol_0 &= \sum_{s=1}^2 (-1)^s \left( \int_{\partial c_s}  \mathring g(X, \mathring \nu_s) \dif \mathrm{L}_0 + \tfrac{2}{r} \int_{c_s} \mathring g( X, \tfrac{\partial}{\partial t}) \dif \vol_0 \right) \\
	&= \sum_{s=1}^2 (-1)^s \, r  \big(C_1 \eps_s + C_1' \eps_s^{3/2} \big)
\end{align*}
by direct computation, where $\eps_s$ is the neck scale parameter associated to the neck adjoining the curve $\partial c_s$, while $C_1$ and $C_1'$ are constants independent of $r$, $\eps$ and $\delta$.  For the second term, the expansion in equation \ref{eqn:geomcexp} implies
\begin{equation}
	\label{eqn:balmapone}
	\begin{aligned}
		I ( \Sigma, X, g)  -  I ( \Sigma, X, \mathring g) &= \int_{\Sigma \cap \mathcal V_k} \mathring g(X, \mathring N) \Big( 
		\big( \tfrac{1}{6} \ric( Y, Y) + \tfrac{1}{12} \bar \nabla_Y \ric(  Y,  Y) \big) \big(\mathring H - \tfrac{2}{r} \big) \\[-0.5ex]
		& \qquad \qquad
		- \big( \tfrac{1}{3} \riem(E_i, Y, E_j,Y) + \tfrac{1}{6} \bar\nabla_Y \riem(E_i, Y, E_j, Y) \big) \mathring B^{ij}  \\
		& \qquad \qquad
		- \tfrac{2}{3} \ric( \mathring N, Y) - \tfrac{1}{2} \bar \nabla_Y \ric( \mathring N, Y) \\
		&\qquad \qquad +  \tfrac{1}{12}  \bar\nabla_{\mathring N} \ric( Y, Y) -  \tfrac{1}{6}  \bar\nabla_{\mathring N} \riem(\mathring N, Y, \mathring N, Y) \Big) \dif \vol_0
		+ \mathcal O(r^5)
	\end{aligned}
\end{equation}
where $Y$ is the position vector field of $\Sigma_k$, while the quantities $g, N, H, B$ and $\mathring g, \mathring N, \mathring H, \mathring B$ have their usual meanings.    Since $\mathcal V_k$ is the normal graph of the function $r G$ over the sphere $S_k$ as in Section \ref{sec:approxsol}, one can replace the integral in \eqref{eqn:balmapone} with an integral over $S_k$, at the expense of an error of size $\mathcal O(\eps | \log (\eps) | r^3)$.  Hence by direct computation using $\mathring H = \frac{2}{r}$ and $\mathring B_{ij} = r \mathring h_{ij}$ one finds
$$I ( \Sigma, X, g)  -  I ( \Sigma, X, \mathring g) = -C_2 r^4 \dot S(p_k) + \mathcal O(\eps | \log (\eps) |  r^3)$$
where $C_2>0$  is a constant independent of $r$, $\eps$ and $\delta$.  An similar computation is performed in \cite{ye}.

It remains to estimate the error terms appearing in \eqref{eqn:balmapzero}.  In the third term, the fact that $H[\Sigma_f] = H[\Sigma] + \mathcal L(rf) + \mathcal Q(rf)$ must be used.  Thus 
$$\big|  I ( \Sigma_f, \tilde X,  g)  -  I ( \Sigma, \tilde X, g) \big| \leq C r^2 |f|_{C^{2, \alpha}_0 (\Sigma \cap \mathcal V_k)} \leq C r^2 R(r, \eps, \delta) $$
using the estimate of $f$ from Proposition \ref{prop:soluptocoker}.  In the fourth term, observe that $\tilde X - X$ is supported in a collar of width $\mathcal O(r\eps^{3/4})$ around the transition regions of $\Sigma \cap \mathcal V_k$.  Hence 
$$ \big| I ( \Sigma, \tilde X, g)  -  I ( \Sigma, X, g)  \big| \leq C r^2 \eps^{3/2} $$
using the estimate from Step 3 of Proposition \ref{prop:meancurvest} for the mean curvature in the transition region.

The derivation of the formula \eqref{eqn:neckbalformula} is similar to what has been done above.  That is, writing $I(\Sigma', X', g')$ as before, but with $X'$ equal to either $X := \frac{\partial }{\partial t}$ or $\tilde X := \chi_{\mathit{neck}, k}^{\tau_{1}} \frac{\partial}{\partial t}$, one finds the same decomposition as \eqref{eqn:balmapzero} for $I(\Sigma_f, \tilde X, g)$.  But now,
\begin{align*}
	\int_{\Sigma \cap \mathcal N^{\tau_{1}}_k} \big( \mathring H[\Sigma] - \tfrac{2}{r} \big) I_k \, \dif \vol_0 &= - \frac{2}{r} \int_{\Sigma \cap \mathcal N^{\tau_{1}}_k} I_k \, \dif \vol_0 = C \delta_k r  \eps_k^{3/2}
\end{align*}
where $\delta_k$ is the displacement parameter of the neck $\mathcal N_k$ and $C$ is a constant independent of $r$, $\eps$ and $\delta$.  This is because $\int_{\Sigma \cap \mathcal N^{\tau_{1}}_k} \mathring H[\Sigma] \, I_k \, \dif \vol_0 = 0$ exactly (this is the first variation formula for the exactly minimal surface $\Sigma \, \cap \, \mathrm{supp}( \chi_{\mathit{neck}, k}^{\tau_{1}})$) and $I_k$ is an odd function with respect to the neck having $\delta_k = 0$, whereas the integral is being taken over the neck with $\delta_k \neq 0$.  The remaining terms in the expansion of $I(\Sigma_f, \tilde X, g)$ are small error terms whose estimates are sufficiently similar to the analogous ones above and will not be repeated for the sake of brevity. 
\end{proof}

\subsection{Proof of the Main Theorem}

The formul\ae\ developed for the balancing map $B_r : \R^{d} \rightarrow \R^{d}$ in the previous section make it possible to choose an exactly CMC surface from amongst the family of surfaces $\mu_{rf_r(\sigma, \delta)} \big( \apsol \big)$.  This will be done as follows.  First, because of Lemma \ref{lemma:projop}, it is sufficient to find $(\sigma, \delta)$ so that the right hand sides of equations \eqref{eqn:neckbalformula} and \eqref{eqn:sphbalformula} vanish for every $k$.   One should realize that $p_k$ in these equations can be expressed in terms of $\eps_1, \eps_2, \ldots$ via the formula $p_k := \gamma( 2 k r + \sum_{l=1}^k \eps_l)$ and $\eps_k$ can be expressed in terms of $\sigma_k$ via the formula $\sigma_k := \Lambda_k(\eps_k)$ as in Section \ref{sec:approxsol}.  (In the case of the one-ended surface, let the relationship $\eps_K = \rho_T(0)$ for $T = 2 + \sigma_K/r$ satisfied by the Delaunay end of $\apsoloe$ be written $\eps_K := \Lambda_K(\sigma_K)$ for consistency.)   Note that $\sigma_k$ and $p_k$ are smooth functions of $\eps_k$.   Finding the appropriate value $(\sigma,\delta)$ will amount to applying the implicit function theorem for smooth maps between finite dimensional spaces to this system of equations, and will lead to a unique solution $(\sigma, \delta ) := (\sigma(r), \delta(r))$ for all sufficiently small $r>0$ and $\eps$, $\delta$ satisfying $r^3 < \eps < r^2$ and $\delta < \eps^{1/2}$.

\paragraph*{The finite-length surface.}  The equations that must be solved to produce the finite-length CMC surface are as follows: if the various error quantities appearing in equations \eqref{eqn:balform}, divided by $r$, are denoted $E_{s,k}(r, \eps)$ where $s = 1$ refers to a neck and $s=2$ refers to a sphere, then
\begin{subequations}
\label{eqn:neckchoice}
\begin{equation}
	\label{eqn:neckchoiceone}
	\begin{aligned}
		0 &= \delta_1 + \eps_1^{-3/2} E_{1,1}(r, \eps, \delta) \\
		& \; \; \vdots \\
		0 &= \delta_{K-1} + \eps_K^{-3/2}  E_{1,K-1}(r, \eps, \delta)
\end{aligned}
\end{equation}
as well as
\begin{equation}
	\label{eqn:neckchoicetwo}
	\begin{aligned}
		0 &= - q_1(\eps_{1}) + q_{2}(\eps_{2}) - C_{2, 1} r^{3} \dot S (p_1) + E_{2,1}(r , \eps , \delta) \\
		& \; \; \vdots \\
		0 &= - q_k(\eps_{k}) + q_{k+1}(\eps_{k+1}) - C_{2, k} r^{3} \dot S (p_k) + E_{2,k}(r , \eps , \delta) \\
		& \; \; \vdots \\
		0 &= -  q_K(\eps_{K})  - C_{2, K} r^{3} \dot S (p_K) + E_{2,K}(r , \eps , \delta)
	\end{aligned}
\end{equation}
\end{subequations}
where $q_k(\eps) := C_{1, k} \eps + C_{1, k}' \eps^{3/2}$ and $p_k := \gamma( 2 k r + \sum_{l=1}^k \eps_l)$ while $C_{1, k}, C_{1, k}'$ and $C_{2, k}$ are various constants independent of $r$, $\eps$ and $\delta$.  Note that the $E_{s, k}$ are smooth functions of $\eps$.  Also, because the $t \mapsto - t$ symmetry that has been imposed since the beginning, the scalar curvature must have a critical point at $p_0$.

One should now view the equations in \eqref{eqn:neckchoice} as a systems of equations for the $\eps$ and $\delta$ variables depending on the parameter $r$ that is to be treated using the implicit function theorem.  When $r= 0$ there is an exact solution $\delta_1 = \cdots = \delta_{K-1} = 0$ and $\eps_1 = \cdots = \eps_K = 0$.  Furthermore, it is easy to see that the derivative matrix of the function $\Phi (\eps, \delta, r)$ defined by the right hand sides of \eqref{eqn:neckchoiceone} and \eqref{eqn:neckchoicetwo} in the $\eps$ and $\delta$ variables is invertible at $r=0$ with a lower bound of size $\mathcal O(1)$ on its determinant (the derivative matrix is upper-triangular with non-zero constants of size $\mathcal O(1)$ on the diagonal).   Hence by the inverse function theorem there is a solution of \eqref{eqn:neckchoice} for all sufficiently small $r$, and the dependence of $\eps$ and $\delta$ on $r$ is smooth.   Note that the solution for small $r$ will have $\eps_k = \mathcal O( r^3 \sum_{k'=0}^k S(p_{k'}))$ and hence $C_1 r^3 \leq \eps_k \leq C_2 r^2$ for numerical constants $C_1$ and $C_2$.  This is because the sum  $ r \sum_{k'=0}^k \dot S(p_{k'}))$ approximates a Riemann sum for the integral of $S$ along $\gamma$ from $t=0$ to $t=2 K r$ and a uniform bound on the oscillation of the scalar curvature of the ambient manifold has been assumed.  Furthermore, it is also the case that $\delta_k < \eps^{1/2}$ for small $r$ simply by examining the dependence of the $E(r, \eps, \delta)$ quantities on its arguments.  This completes the construction of the finite-length CMC surface. \hfill \qedsymbol

\paragraph*{The one-ended surface.} The equations that must be solved to produce the one-ended CMC surface are slightly different.  Using the same notation as above, these equations are
\begin{subequations}
\label{eqn:neckchoiceoe}
\begin{equation}
	\label{eqn:neckchoiceoneoe}
	\begin{aligned}
		0 &= \delta_1 + \eps_1^{-3/2} E_{1,1}(r, \eps, \delta) \\
		& \; \; \vdots \\
		0 &= \delta_{K-1} + \eps_K^{-3/2}  E_{1,K-1}(r, \eps, \delta)
\end{aligned}
\end{equation}
as well as
\begin{equation}
	\label{eqn:neckchoicetwooe}
	\begin{aligned}
		0 &= q_0(\eps_0) - C_{2, 0} r^3 \dot S(p_0) + E_{2, 0} (r, \eps, \delta) \\
		0 &= - q_0(\eps_{0}) + q_{1}(\eps_{1}) - C_{2, 1} r^{3} \dot S (p_1) + E_{2,1}(r , \eps , \delta) \\
		& \; \; \vdots \\
		0 &= - q_{K-1}(\eps_{k}) + q_{K}(\eps_{K}) - C_{2, K} r^{3} \dot S (p_K) + E_{2,K}(r , \eps , \delta) \, .
	\end{aligned}
\end{equation}
\end{subequations}

One should again view the equations in \eqref{eqn:neckchoiceoe} as a systems of equations for the $\eps$ and $\delta$ variables to be treated using the implicit function theorem, but this time depending on the parameters $r$ and the point $p_0$.  When $r= 0$ and $p_0$ is any point on $\gamma$, there is an exact solution $\delta_1 = \cdots = \delta_{K-1} = 0$ and $\eps_1 = \cdots = \eps_K = 0$.  Furthermore, the derivative matrix of the function of $(\eps, \delta, r)$ defined by the right hand sides of \eqref{eqn:neckchoiceoe} in the $(\eps,\delta)$ variables is invertible at $r=0$ with a lower bound of size $\mathcal O(1)$ on its determinant.   Hence by the inverse function theorem there is a solution of \eqref{eqn:neckchoice} for all sufficiently small $r$, the dependence of $\eps$ and $\delta$ on $r$ is smooth, and the dependence of the solution on $r$ is the same as before.  This completes the construction of the one-ended CMC surface. \hfill \qedsymbol

\bigskip

\renewcommand{\baselinestretch}{1}
\small

\bibliography{sc}

\providecommand{\bysame}{\leavevmode\hbox to3em{\hrulefill}\thinspace}
\providecommand{\MR}{\relax\ifhmode\unskip\space\fi MR }
\providecommand{\MRhref}[2]{%
  \href{http://www.ams.org/mathscinet-getitem?mr=#1}{#2}
}
\providecommand{\href}[2]{#2}
\begin{thebibliography}{10}

\bibitem{mepacard1}
A.~Butscher and F.~Pacard, \emph{Doubling constant mean curvature tori in
  {$S\sp 3$}}, Ann. Sc. Norm. Super. Pisa Cl. Sci. (5) \textbf{5} (2006),
  no.~4, 611--638.

\bibitem{mepacard2}
\bysame, \emph{Generalized doubling constructions for constant mean curvature
  hypersurfaces in {$S\sp {n+1}$}}, Ann. Global Anal. Geom. \textbf{32} (2007),
  no.~2, 103--123.

\bibitem{gbkkrs}
K.~Grosse-Brauckmann, N.~Korevaar, R.~Kusner, J.~Ratzkin, and J.~Sullivan,
  \emph{Coplanar k-unduloids are nondegenerate}, Preprint: arXiv:0712.1865.

\bibitem{gks}
K.~Grosse-Brauckmann, R.~Kusner, and J.~Sullivan, \emph{Triunduloids: embedded
  constant mean curvature surfaces with three ends and genus zero}, J. Reine
  Angew. Math. \textbf{564} (2003), 35--61.

\bibitem{gks2}
\bysame, \emph{Coplanar constant mean curvature surfaces}, Comm. Anal. Geom.
  \textbf{15} (2007), no.~5, 985--1023. \MR{MR2403193}

\bibitem{kapouleas7}
Nikolaos Kapouleas, \emph{Complete constant mean curvature surfaces in
  {E}uclidean three-space}, Ann. of Math. (2) \textbf{131} (1990), no.~2,
  239--330.

\bibitem{kapouleas6}
\bysame, \emph{Compact constant mean curvature surfaces in {E}uclidean
  three-space}, J. Differential Geom. \textbf{33} (1991), no.~3, 683--715.

\bibitem{kks}
N.~Korevaar, R.~Kusner, and B.~Solomon, \emph{The structure of complete
  embedded surfaces with constant mean curvature}, J. Diff. Geom. \textbf{30}
  (1989), 465--503.

\bibitem{mazzeopacardends}
R.~Mazzeo and F.~Pacard, \emph{Constant mean curvature surfaces with {D}elaunay
  ends}, Comm. Anal. Geom. \textbf{9} (2001), no.~1, 169--237.

\bibitem{mazzeopacardtubes}
\bysame, \emph{Foliations by constant mean curvature tubes}, Comm. Anal. Geom.
  \textbf{13} (2005), no.~4, 633--670.

\bibitem{mazzeopacardpollack}
R.~Mazzeo, F.~Pacard, and D.~Pollack, \emph{Connected sums of constant mean
  curvature surfaces in {E}uclidean 3 space}, J. Reine Angew. Math.
  \textbf{536} (2001), 115--165.

\bibitem{mazzeosurvey}
Rafe Mazzeo, \emph{Recent advances in the global theory of constant mean
  curvature surfaces}, Noncompact problems at the intersection of geometry,
  analysis, and topology, Contemp. Math., vol. 350, Amer. Math. Soc.,
  Providence, RI, 2004, pp.~179--199.

\bibitem{meeks}
William~H. Meeks, III, \emph{The topology and geometry of embedded surfaces of
  constant mean curvature}, J. Differential Geom. \textbf{27} (1988), no.~3,
  539--552.

\bibitem{pacardxu}
F.~Pacard and X.~Xu, \emph{Constant mean curvature spheres in {R}iemannian
  manifolds}, preprint:
  http://perso-math.univ-mlv.fr/users/pacard.frank/PR-6.pdf.

\bibitem{pacardnotes}
Frank Pacard, \emph{Connected {S}um {C}onstructions in {G}eometry and
  {N}onlinear {A}nalysis}, preprint:
  http://perso-math.univ-mlv.fr/users/pacard.frank/Lecture-Part-I.pdf.

\bibitem{pacardsurvey}
\bysame, \emph{Surfaces \`a courbure moyenne constante}, Image des
  math\'ematiques 2006, Publications of the CNRS, CNRS, Paris, 2006,
  pp.~107--112.

\bibitem{ratzkinthesis}
Jesse Ratzkin, \emph{An {E}nd-to-{E}nd {C}onstruction for {C}onstant {M}ean
  {C}urvature {S}urfaces}, Ph.D. thesis, University of Washington, 2001,
  preprint: \\ http://www.math.uga.edu/$\sim$jratzkin/papers/thesis.pdf.

\bibitem{rosenberg}
Harold Rosenberg, \emph{Constant mean curvature surfaces in homogeneously
  regular 3-manifolds}, Bull. Austral. Math. Soc. \textbf{74} (2006), no.~2,
  227--238.

\bibitem{schoen}
Richard~M. Schoen, \emph{The existence of weak solutions with prescribed
  singular behavior for a conformally invariant scalar equation}, Comm. Pure
  Appl. Math. \textbf{41} (1988), no.~3, 317--392.

\bibitem{ye}
Rugang Ye, \emph{Foliation by constant mean curvature spheres}, Pacific J.
  Math. \textbf{147} (1991), no.~2, 381--396.

\end{thebibliography}
\bibliographystyle{amsplain}

\end{document}